\documentclass[review]{elsarticle}

\usepackage{lineno,hyperref}
\RequirePackage{amsmath,amsfonts,amsthm}
\usepackage{float}
\usepackage{geometry}
\usepackage{subfigure}
\usepackage{graphicx}
\usepackage{epstopdf} 
\usepackage{latexsym}
\usepackage{color}
\usepackage{mathptmx}      
\allowdisplaybreaks \allowdisplaybreaks[4]
\newtheorem{ap}{Assumption}[section]
\newtheorem{theorem}{Theorem}[section]
\newtheorem{lemma}{Lemma}[section]
\newtheorem{definition}{Definition}[section]
\newtheorem{remark}{Remark}[section]
\newtheorem{corollary}{Corollary}[section]
\newtheorem{proposition}{Proposition}[section]
\newtheorem{problem}{Problem}
\newtheorem{example}{Example}[section]
\modulolinenumbers[5]

\begin{document}

\begin{frontmatter}

\title{Stochastic modified equations for symplectic methods applied to rough Hamiltonian systems based on the Wong--Zakai approximation}

\author[mymainaddress,mysecondaryaddress]{Chuchu Chen}
\ead{chenchuchu@lsec.cc.ac.cn}

\author[mymainaddress,mysecondaryaddress]{Jialin Hong}
\ead{hjl@lsec.cc.ac.cn}

\author[mymainaddress,mysecondaryaddress]{Chuying Huang\corref{mycorrespondingauthor}}
\cortext[mycorrespondingauthor]{Corresponding author}
\ead{huangchuying@lsec.cc.ac.cn}

\address[mymainaddress]{LSEC, ICMSEC, Academy of Mathematics and Systems Science, Chinese Academy of Sciences, Beijing 100190, China}
\address[mysecondaryaddress]{School of Mathematical Sciences, University of Chinese Academy of Sciences, Beijing 100049, China}

\begin{abstract}
We investigate the stochastic modified equation which plays an important role in the stochastic backward error analysis for explaining the mathematical mechanism of a numerical method. The contribution of this paper is threefold. First, we construct a new type of stochastic modified equation, which is a perturbation of the Wong--Zakai approximation of the rough differential equation. For a symplectic method applied to a rough Hamiltonian system, the associated stochastic modified equation is proved to have a Hamiltonian formulation. Second, the pathwise convergence order of the truncated modified equation to the numerical method is obtained by techniques in the rough path theory. Third, if increments of noises are simulated by truncated random variables, we show that the one-step error can be made exponentially small with respect to the time step size. Numerical experiments verify our theoretical results.
\end{abstract}


\end{frontmatter}

\linenumbers

\section{Introduction}
In the study of a numerical method for a deterministic ordinary differential equation, the modified equation whose solution coincides with the numerical solution is crucial in the backward error analysis. It gives a lot of insights into the numerical method, especially for illustrating the long time superiority of symplectic methods for Hamiltonian systems. More precisely, the modified equation associated to a symplectic method is proved to be a perturbed Hamiltonian system, which reveals the mechanism of the symplectic method over long time simulation. The readers are referred to the monograph \cite{GeometricEH} for a detailed review.

For the stochastic differential equation (SDE) driven by the standard Brownian motion
\begin{align*}
dY_t=V(Y_t)dW_t,
\end{align*}
there exist various types of stochastic modified equations in different senses of convergence.
In view of the weak convergence, adding the modified coefficient with powers of the time step size $h$ to the original SDE yields a modified equation of the form
\begin{align}\label{mod1}
d\tilde{Y}_t=\big[V(\tilde{Y}_t)+\tilde{V}(\tilde{Y}_t)h^p\big]dW_t,
\end{align}
which fits the numerical method to a higher weak order.
The modified coefficient $\tilde{V}$ can be determined by the weak Taylor expansion \cite{S06} or by the expansion of the backward Kolmogorov equation \cite{Z11}.
As an application, the first order integrated Euler method is proposed for the stochastic Langevin equation in \cite{Z11} to preserve the mean of a modified Hamiltonian.
Another application of this kind of modified equations is to construct high weak order methods; see \cite{ACVZ12,HSW17}.
The modification is also considered
at the level of the generator associated with the process solution of the SDE instead of at the level of the SDE, which leads to the modified Kolmogorov equation
\begin{align*}
\frac{\partial \tilde{u}}{\partial t}=\big[\mathcal{L}+\mathcal{L}_1h+\cdots+\mathcal{L}_Nh^N\big]\tilde{u}.
\end{align*}
Based on the modified Kolmogorov equation, \cite{DF12} proves that
the numerical solution obtained by the Euler method for SDEs on the torus is exponentially mixing up to negligible terms. 
The results are extended to implicit methods for SDEs on $\mathbb{R}^m$ in \cite{2019BIT,K15BIT,K15IMA}. 
With respect to strong convergence, using multiple Stratonovich integrals $J_{\alpha,t}$, 
\cite{D16} defines the modified equation 
\begin{align*}
d\tilde{Y}_t=\big[V(\tilde{Y}_t)+\sum_{\alpha}\tilde{V}_\alpha(\tilde{Y}_t)J_{\alpha,t}\big]dW_t
\end{align*}
for the Euler method, and the optimal truncation of the above series is studied.

As fundamental models in many physical and engineering sciences, such as the passive tracer model and the Kubo oscillator, the phase flows of stochastic Hamiltonian systems preserve the symplectic structure pathwisely and there have been a great amount of work about the construction of stochastic symplectic methods after the pioneering articles \cite{Milstein,additive}. Lots of numerical simulations have shown that the stochastic symplectic methods are superior over long time computation to non-symplectic ones. 
From the perspective of the stochastic modified equation to investigate the superiority of the stochastic symplectic methods,
it is natural to ask:
\begin{problem}\label{P1}
	For a stochastic symplectic method applied to a stochastic Hamiltonian system, does there exist a stochastic modified equation which has a stochastic Hamiltonian formulation, such that its exact solution coincides with the numerical solution?
\end{problem}

This problem is partially solved by \cite{WHS16BIT,WXZ}. As far as the weak convergence is concerned, for the cases of that the Hamiltonians in diffusion parts do not depend on the generalized coordinate and momenta simultaneously, the modified equations in form of \eqref{mod1} for stochastic symplectic methods are derived in \cite{WHS16BIT} via the generating function. These modified equations are perturbed stochastic Hamiltonian systems with respect to the original systems. In \cite{WXZ}, the modified coefficient in \eqref{mod1} is deduced for a symplectic splitting method applied to separable Hamiltonian systems with additive noises, and the flow of the corresponding modified equation preserve the symplectic structure.

In this article, we investigate Problem \ref{P1} for stochastic symplectic methods applied to stochastic Hamiltonian systems driven by Gaussian rough paths proposed in \cite{HHW18}:
\begin{align}\label{eqX}
dY_t=V(Y_t)dX_t,
\end{align}
where $X$ is more general than standard Brownian motions.
The numerical study of the rough differential equations has drawn a lot of attentions (see e.g., \cite{Deya,CN2017,Kellyaap,LT2019}).
Given the time step size $h$ and the numerical solution $\{Y^h_n\}^N_{n=1}$, 
denoting by $x^h$ the piecewise linear approximation of $X$, we present a new type of stochastic modified equation 
\begin{align*}
d\tilde{y}_t=\big[V(\tilde{y}_t)+\sum_{\alpha}\tilde{V}_\alpha(\tilde{y}_t)(X^{1}_{t_{n},t_{n+1}})^{\alpha_1}\cdots(X^{d}_{t_{n},t_{n+1}})^{\alpha_d}\big]dx^h_t,\quad t\in[t_n,t_{n+1}],
\end{align*}
which is a perturbation for the Wong--Zakai approximation of equation \eqref{eqX} and satisfies that $\tilde{y}_{t_n}=Y^h_n$. We refer to \cite{MLMC16,Wzk} for the convergence analysis on the Wong--Zakai approximation.
Based on the orthogonal polynomials with respect to the measure induced by the increments of noises, we prove that if a symplectic method is applied to a rough Hamiltonian system, then for any $\alpha$, there exists a Hamiltonian $\mathcal{H}_\alpha$ such that 
\begin{align*}
\tilde{V}_\alpha=\mathbb{J}^{-1}\nabla \mathcal{H}_\alpha.
\end{align*}
This implies that stochastic modified equations for symplectic methods are also stochastic Hamiltonian systems, and gives a positive answer to Problem \ref{P1}. 

Since the coefficient of the stochastic modified equation is an infinite series, 
the truncated modified equation should be taken into consideration as well. Further problems are:

\begin{problem}\label{P2}
	What is the convergence rate of the error between the numerical solution and the exact solution of the truncated modified equation?
\end{problem}

\begin{problem}\label{P3}
Can the error be made exponentially small with respect to the time step size?
\end{problem}

Using the It\^o--Lyons map in the rough path theory, we obtain the pathwise convergence rate of the exact solution $\tilde{y}^{\tilde{N}}$ of truncated modified equation to the numerical solution $Y^{h}$, that is,
\begin{align*}
\sup_{1\le n\le N}\|\tilde{y}^{\tilde{N}}_{t_n}-Y^{h}_{n}\|\le C(\omega)h^{\frac{\tilde{N}+1}{p}-1},\quad a.s.,
\end{align*}
where $\tilde{N}$ is the truncation number and $p$ depends on the regularity of the driving signal. This convergence result answers Problem \ref{P2}. 
For Problem \ref{P3}, we focus on the case for the standard Brownian motion where the increments of noises are simulated by truncated random variables proposed in \cite{Milstein}. Due to the lack of explicit expansion formulas of implicit numerical methods, we use the analytic assumption to estimate the numerical solution, the modified equation and the truncated modified equation, successively. Combining the estimates yields that there exists some truncation number $\tilde{N}=\tilde{N}(h)$ such that the one-step error  is exponentially small with respect to the time step size:
\begin{align*}
\|\tilde{y}^{\tilde{N}}_{t_1}-Y^h_1\|\le Che^{- h_0/h^{\frac12-\epsilon} }.
\end{align*}

The rest of this article is organized as follows. In Section \ref{sec2}, we introduce the basic results in the rough path theory. In Section \ref{sec3}, for Problem \ref{P1}, we illustrate the procedure in constructing stochastic modified equations based on the Wong--Zakai approximation and prove that stochastic modified equations associated to stochastic symplectic methods are Hamiltonian systems as well. In Section \ref{sec4}, we prove the pathwise convergence rate of the error between the numerical solution and the exact solution of truncated modified equation, and give the exponentially convergence for one-step error in the case of truncated noises, which answers Problems \ref{P2}-\ref{P3}. Numerical experiments are presented in Section \ref{sec5} to support theoretical results.

\section{Preliminaries}\label{sec2}

In this section, we review the well-posedness of stochastic differential equations in the sense of rough path theory (\cite{Friz,Lyons}).

Consider the stochastic differential equation driven by multi-dimensional Gaussian signal
\begin{equation}\label{sdeH1}
\left\{
\begin{aligned}
dY_{t}&=V_0(Y_{t})dt+\sum^{d}_{l=1}V_{l}(Y_{t})dX^{l}_{t},\quad t\in[0,T];\\
Y_{0}&=z\in\mathbb{R}^{m}.
\end{aligned}
\right.
\end{equation}
For a convenient notation involving the drift term, we define $V:=(V_0,V_1,\cdots,V_d)$, $X:=(X^0,X^1,\cdots,X^d)$, $X^0_t:=t$ and then an equivalent form of equation \eqref{sdeH1} is 
\begin{equation}\label{sdeH}
\left\{
\begin{aligned}
dY_{t}&=V(Y_{t})dX_{t},\quad t\in[0,T];\\
Y_{0}&=z.
\end{aligned}
\right.
\end{equation}
In this article, we focus on the case that the driving signal $X$ satisfies the following assumption.

\begin{ap}\label{R}
	Let $X^{l}:[0,T]\rightarrow\mathbb{R}$, $l=1,\cdots,d$ be independent centered Gaussian processes with continuous sample paths.
	There exist some $\rho \in [1,2)$ and $K \in (0,+ \infty)$ such that the covariance of $X$ satisfies
	\begin{align*}
	\sup_{\{t_{k}\},\{t_{i}\} \in \mathcal{D}([s,t])}\left(\sum_{t_{k},t_{i}}\Big|\mathbb{E} \big[X^l_{t_{k},t_{k+1}} X^l_{t_{i},t_{i+1}}\big]\Big|^{\rho}\right)^{1/ \rho} \leq K |t-s|^{1/ \rho},\quad \forall~0\leq s<t\leq T,
	\end{align*}
	where $\mathcal{D}([s,t])$ denotes the set of all dissections of $[s,t]$ and $X^l_{t_{k},t_{k+1}}:=X^l_{t_{k+1}}-X^l_{t_{k}}$.
\end{ap}

For instance, one can check that fractional Brownian motions with Hurst parameter $H \in (\frac14,\frac12]$, whose covariance is
$\mathbb{E} \big[|X^l_{s,t} |^2\big]= |t-s|^{2H}$,
 satisfy Assumption \ref{R} with $\rho=\frac{1}{2H}$. Since the Kolmogorov continuity theorem shows that the trajectory of the fractional Brownian motion has $(H-\epsilon)$-H\"older regularity with $\epsilon$ being an arbitrarily small positive number, the well-posedness of equation \eqref{sdeH} fails to be established in the Riemann--Stieltjes integral sense. Hence   we interprete \eqref{sdeH} in the rough path sense. To this end, we introduce some basic concepts in the rough path theory (see \cite{Friz} for more details).

Let $p \in [1,\infty)$ and $[p]$ be the integer part of $p$, i.e., $[p]\in\mathbb{N}_+$ with $p-1< [p]\leq p$.
We denote by $\left(G^{[p]}(\mathbb{R}^{d+1}),{\rm d}\right)$ the free step-$[p]$ nilpotent Lie group of $\mathbb{R} ^{d+1}$ equipped with the Carnot--Carath\'eodory metric (\cite[Chap.  7]{Friz}). 
A continuous map $\textbf{X}:[0,T]\rightarrow G^{[p]}(\mathbb{R}^{d+1})\subset \bigoplus _{n=0}^{[p]}(\mathbb{R}^{d+1})^{\otimes n}$ is called $p$-rough path if 
\begin{equation*}
\|\textbf{X}\|_{p\text{-}var;[0,T]}:=\sup_{\{t_{k}\} \in \mathcal{D}([0,T])}\left(\sum_{t_{k}} {\rm d}( \textbf{X}_{t_{k}},\textbf{X}_{t_{k+1}})^{p}\right)^{1/p}<\infty,
\end{equation*}
where $\mathcal{D}([0,T])$ is the set of dissections of $[0,T]$.
Furthermore, we say that $\textbf{X}$ is of H$\rm \ddot{o}$lder-type if
\begin{equation*}
\|\textbf{X}\|_{\frac1p\text{-} {\rm H\ddot{o}l};[0,T]}:=\sup_{0\leq s<t\leq T}\dfrac{{\rm d}(\textbf{X}_{s},\textbf{X}_{t})}{|t-s|^{1/p}} < \infty.
\end{equation*}

For example, if $x:[0,T]\rightarrow \mathbb{R}^{d+1}$ is a function of bounded variation and $x_0=0$, the corresponding rough path can be defined by $S_{[p]}(x):[0,T]\rightarrow G^{[p]}(\mathbb{R}^{d+1})$ with
\begin{equation*}
S_{[p]}(x)_{t}=\left(1,\int_{0\leq u_{1}\leq t}dx_{u_{1}},\cdots,\int_{0\leq u_{1}<\cdots<u_{[p]}\leq t}dx_{u_{1}}\otimes\cdots \otimes dx_{u_{[p]}}\right).
\end{equation*}
It is a canonical lift for $x$ in the sense that the projection of $S_{[p]}(x)$ onto $\mathbb{R}^{d+1}$ coincides with $x$. 

Moreover, 
the Gaussian process $X$ under Assumption \ref{R} can be lifted to a H$\rm \ddot{o}$lder-type $p$-rough path $\textbf{X}\in G^{[p]}(\mathbb{R}^{d+1})$ for any $p > 2\rho$ (\cite[Theorem 15.33]{Friz}), which is defined by the limit of $\{S_3(x^n)\}_{n=1}^{\infty}$ with $\{x^n\}_{n=1}^{\infty}$ being a sequence of piecewise linear or mollifier approximations to $X$. As a consequence, the well-posedness of equation \eqref{sdeH} is given by that of the rough differential equation (RDE)
\begin{equation}\label{rde}
\left\{
\begin{aligned}
dY_{t}&=V(Y_{t})d\textbf{X}_{t},\quad t\in[0,T];\\
Y_{0}&=z.
\end{aligned}
\right.
\end{equation}
In the sequel, we introduce the definition of the solution of equation \eqref{rde} and state the condition for the existence and uniqueness of the solution. Throughout the rest of this paper, we denote by $\|\cdot\|$ the Euclidean norm and by $C$ a generic constant which may be different from line to line.

\begin{definition}(\cite[Definition 10.17]{Friz})\label{solution}
	Let $p\in [1,\infty)$ and $\textbf{X}$ be a $p$-rough path. Suppose that there exists a sequence of functions $\{x^{n}\}_{n=1}^{\infty}$ of bounded variation taking values in $\mathbb{R}^{d+1}$ such that
	\begin{equation*}\label{xn}
	\sup_{n\in \mathbb{N}}\|S_{[p]}(x^{n})\|_{p\text{-}var;[0,T]} < \infty\quad {\rm and}\quad \lim_{n\rightarrow \infty} \sup_{0\leq s<t\leq T} {\rm d}\big(S_{[p]}(x^{n})_{s,t},\textbf{X}_{s,t}\big)=0,
	\end{equation*}
	where $S_{[p]}(x^{n})_{s,t}:=S_{[p]}(x^{n})_{s}^{-1}\otimes S_{[p]}(x^{n})_{t}$ and $\textbf{X}_{s,t}:=\textbf{X}_{s}^{-1}\otimes\textbf{X}_{t}$.
	Suppose in addition that $\{y^{n}\}_{n=1}^{\infty}$ are solutions of equations $dy^{n}_{t}=V(y^{n}_{t})dx^{n}_{t}$, in the Riemann--Stieltjes integral sense, with the same initial value $z$ as in \eqref{rde}. If $y^{n}_{t}$ converges to $Y_{t}$ in the $L^{\infty}([0,T])$-norm, i.e.,
	\begin{align*}
	\lim_{n\rightarrow \infty}\sup_{0\le t\le T}\|y^n_t-Y_t\|=0,
	\end{align*}
	 then we call $Y_{t}$ a solution of \eqref{rde}.
\end{definition}

\begin{definition}(\cite[Definition 10.2]{Friz})\label{Lip}
	Let $\gamma>0$, and $\lfloor \gamma \rfloor$ be the largest integer strictly smaller than $\gamma$, i.e., $\gamma-1\leq\lfloor \gamma \rfloor<\gamma$. We say that $V \in Lip^{\gamma}$ , if $V$ is $\lfloor \gamma \rfloor$ times continuously differentiable and there exists some constant $C$ such that 
	\begin{equation*}
	\begin{split}
	&\|D^{k}V(y)\|\leq C,\quad \forall\ k=0,\cdots, \lfloor \gamma \rfloor,~\forall\ y \in \mathbb{R}^{m},\\
	&\|D^{\lfloor \gamma \rfloor }V(y_{1})-D^{\lfloor \gamma \rfloor }V(y_{2})\|\leq C \|y_{1}-y_{2}\|^{\gamma - \lfloor \gamma \rfloor},\quad \forall\ y_{1},y_{2}\in \mathbb{R}^{m},
	\end{split}
	\end{equation*}
	where $D^{k}V$ denotes $k$th derivative of $V$. 
	The smallest constant $C$ satisfying the above inequalities is denoted by $\|V\|_{Lip^{\gamma}}$.
\end{definition}

\begin{lemma}(\cite[Theorem 10.26 and Theorem 11.6]{Friz})\label{well}
	Let $p\in [1,\infty)$ and $\textbf{X}$ be a $p$-rough path. If $V\in Lip^{\gamma}$ with $\gamma>p$, or $V$ is linear, then \eqref{rde} has a unique solution. Additionaly, the Jacobian $\frac{\partial Y_{t}}{\partial z}$ exists and satisfies the linear RDE
	\begin{equation*}
	\left\{
	\begin{aligned}
	d\frac{\partial Y_{t}}{\partial z}&=\sum^{d}_{l=0}DV_{l}(Y_{t})\frac{\partial Y_{t}}{\partial z}d\textbf{X}^{l}_{t},\quad t\in[0,T];\\
	\frac{\partial Y_{0}}{\partial z}&=\mathbb{I}_{m}\in \mathbb{R}^{m\times m},
	\end{aligned}
	\right.
	\end{equation*}
	where $\mathbb{I}_{m}$ is the identity matrix.
\end{lemma}

\begin{remark}
	If $X$ is the standard Brownian motion, the solution $Y$ of \eqref{rde} solves the corresponding Stratonovich SDE almost surely (\cite[Theorem 17.3]{Friz}).
\end{remark}

\section{Construction of the stochastic modified equation}\label{sec3}
In this section, we propose an approach to deduce the stochastic modified equations for a numerical method under assumptions. The stochastic modified equation is proved to possess the symplectic conservation law if it is associated with a stochatic symplectic method for a rough Hamiltonian system. This answers Problem \ref{P1} proposed in the introduction.

\subsection{Construction of the stochastic modified equations for general methods}\label{sec3.1}
Fix the time step size $h=T/N$, $N\in\mathbb{N}_+$. 
Let $Y^h_n$ be the numerical solution given by certain numerical method, which is an approximation for $Y_{t_n}$, where $t_n=nh$, $n=0,\cdots,N$.
Our main assumption on the numerical method is that $Y^h_{n+1}$ can be expanded as an infinite series of functions of $Y^h_n$:
\begin{align}\label{expYn}
Y^h_{n+1}=Y^h_n+\sum_{|\alpha|=1}^{\infty}d_\alpha(Y^h_n)h^{\alpha_0}(X^{1}_{t_{n},t_{n+1}})^{\alpha_1}\cdots(X^{d}_{t_{n},t_{n+1}})^{\alpha_d},
\end{align}
where $\alpha=(\alpha_0,\cdots,\alpha_d)\in\mathbb{N}^{d+1}$, $|\alpha|:=\alpha_0+\cdots+\alpha_d\ge 1$ and $Y^h_0=z$. 
In addition, for $|\alpha|=1$ with $\alpha_l=1$, we define $V_\alpha(y):=V_l(y)$. A necessary the condition for that $Y^h_n$ converging to $Y_{t_n}$ pathwisely is that there exists some $\alpha_l\in\{1,\cdots,d\}$ with $\alpha_l=1$ and $|\alpha|=1$ such that
\begin{align}\label{ap-f1}
d_\alpha(y)=V_\alpha(y).
\end{align}

A typical example is the $s$-stage Runge--Kutta (RK) method 
\begin{equation}\label{RK}
\left\{
\begin{aligned}
Y^h_{n+1,i}&=Y^h_{n}+\sum^{s}_{j=1}a_{ij}\Big(V_0(Y^h_{n+1,j})h+\sum_{l=1}^{d}V_l(Y^h_{n+1,j})X^l_{t_{n},t_{n+1}}\Big),\\
Y^h_{n+1}&=Y^h_{n}+\sum^{s}_{i=1}b_{i}\Big(V_0(Y^h_{n+1,i})h+\sum_{l=1}^{d}V_l(Y^h_{n+1,i})X^l_{t_{n},t_{n+1}}\Big).
\end{aligned}
\right.
\end{equation}
Then the Taylor expansion produces that for $l=0,\cdots,d$,
\begin{align*}
V_l(Y^h_{n+1,i})
=&V_l(Y^h_{n})+V'_l(Y^h_{n})\Big(\sum^{s}_{j=1}a_{ij}\Big(V_0(Y^h_{n+1,j})h+\sum_{l_1=1}^{d}V_{l_1}(Y^h_{n+1,j})X^{l_1}_{t_{n},t_{n+1}}\Big)\Big)\\
&+\frac12 V''_l(Y^h_{n})\Big(\sum^{s}_{j=1}a_{ij}\Big(V_0(Y^h_{n+1,j})h+\sum_{l_1=1}^{d}V_{l_1}(Y^h_{n+1,j})X^{l_1}_{t_{n},t_{n+1}}\Big)\Big)^2+\cdots\\
=&V_l(Y^h_{n})+\sum^{s}_{j=1}a_{ij}V'_l(Y^h_{n})\Big(V_0(Y^h_{n})h+\sum_{l_1=1}^{d}V_{l_1}(Y^h_{n})X^{l_1}_{t_{n},t_{n+1}}\Big)\\
&+\sum^{s}_{j_1,j_2=1}a_{ij_1}a_{ij_2}V'_l(Y^h_{n})
\Big(V'_0(Y^h_{n})h+\sum_{l_1=1}^{d}V'_{l_1}(Y^h_{n})X^{l_1}_{t_{n},t_{n+1}}\Big)
 \Big(V_0(Y^h_{n})h+\sum_{l_2=1}^{d}V_{l_2}(Y^h_{n})X^{l_2}_{t_{n},t_{n+1}}\Big)\\
&+\frac12 \sum^{s}_{j_1,j_2=1}a_{ij_1}a_{ij_2}V''_l(Y^h_{n})\Big(V_0(Y^h_{n})h+\sum_{l_1=1}^{d}V_{l_1}(Y^h_{n})X^{l_1}_{t_{n},t_{n+1}}\Big)\Big(V_0(Y^h_{n})h+\sum_{l_2=1}^{d}V_{l_2}(Y^h_{n})X^{l_2}_{t_{n},t_{n+1}}\Big)+\cdots.
\end{align*}
Here $V'_0(y)V_0(y)$ denotes that the derivative of $V_0(y)$ acts on $V_0(y)$, and $V''_0(y)V_0(y)V_0(y)$ is the second derivative of $V''_0(y)$ acting $(V_0(y),V_0(y))$. Other operators are defined similarly.
Substituting them into \eqref{RK}, we get 
\begin{align*}
Y^h_{n+1}
=&Y^h_{n}+\sum^{s}_{i=1}b_{i}\bigg[
V_0(Y^h_{n})+\sum^{s}_{j=1}a_{ij}V'_0(Y^h_{n})\Big(V_0(Y^h_{n})h+\sum_{l_1=1}^{d}V_{l_1}(Y^h_{n})X^{l_1}_{t_{n},t_{n+1}}\Big)\\
&+\sum^{s}_{j_1,j_2=1}a_{ij_1}a_{ij_2}V'_0(Y^h_{n})
\Big(V'_0(Y^h_{n})h+\sum_{l_1=1}^{d}V'_{l_1}(Y^h_{n})X^{l_1}_{t_{n},t_{n+1}}\Big)
\Big(V_0(Y^h_{n})h+\sum_{l_2=1}^{d}V_{l_2}(Y^h_{n})X^{l_2}_{t_{n},t_{n+1}}\Big)\\
&+\frac12 \sum^{s}_{j_1,j_2=1}a_{ij_1}a_{ij_2}V''_0(Y^h_{n})\Big(V_0(Y^h_{n})h+\sum_{l_1=1}^{d}V_{l_1}(Y^h_{n})X^{l_1}_{t_{n},t_{n+1}}\Big)\Big(V_0(Y^h_{n})h+\sum_{l_2=1}^{d}V_{l_2}(Y^h_{n})X^{l_2}_{t_{n},t_{n+1}}\Big)\bigg]h\\
&+\sum^{s}_{i=1}\sum_{l=1}^{d}b_{i}
\bigg[V_l(Y^h_{n})+\sum^{s}_{j=1}a_{ij}V'_l(Y^h_{n})\Big(V_0(Y^h_{n})h+\sum_{l_1=1}^{d}V_{l_1}(Y^h_{n})X^{l_1}_{t_{n},t_{n+1}}\Big)\\
&+\sum^{s}_{j_1,j_2=1}a_{ij_1}a_{ij_2}V'_l(Y^h_{n})
\Big(V'_0(Y^h_{n})h+\sum_{l_1=1}^{d}V'_{l_1}(Y^h_{n})X^{l_1}_{t_{n},t_{n+1}}\Big)
\Big(V_0(Y^h_{n})h+\sum_{l_2=1}^{d}V_{l_2}(Y^h_{n})X^{l_2}_{t_{n},t_{n+1}}\Big)\\
&+\frac12 \sum^{s}_{j_1,j_2=1}a_{ij_1}a_{ij_2}V''_l(Y^h_{n})\Big(V_0(Y^h_{n})h+\sum_{l_1=1}^{d}V_{l_1}(Y^h_{n})X^{l_1}_{t_{n},t_{n+1}}\Big)\Big(V_0(Y^h_{n})h+\sum_{l_2=1}^{d}V_{l_2}(Y^h_{n})X^{l_2}_{t_{n},t_{n+1}}\Big)\bigg]X^l_{t_{n},t_{n+1}}\\
&+\cdots.
\end{align*}

To search the modified equation such that $Y^h_n$ solves exactly at $t_n$, we start from the Wong--Zakai approximation of equation \eqref{sdeH}, i.e., 
\begin{equation}\label{Wzk}
\left\{
\begin{aligned}
dy^{h}_{t}&=V(y^{h}_{t})dx^{h}_{t},\quad t\in[0,T];\\
y^{h}_{0}&=z,
\end{aligned}
\right.
\end{equation}
where $x^h=(x^{h,0},x^{h,1},\cdots,x^{h,d})$ is the piecewise linear approximation to $X$ with 
\begin{align}\label{wzkX}
x^{h,l}_{t}:= X^{l}_{t_{n}}+\dfrac{t-t_{n}}{h}X^{l}_{t_{n},t_{n+1}},\quad  t \in (t_{n},t_{n+1}],~n=0,\cdots,N-1.
\end{align}
Based on the fact that the Wong--Zakai approximation  \eqref{Wzk} is also a random differential equation
\begin{equation}\label{Wzk1}
\left\{
\begin{aligned}
\dot{y}^{h}_{t}&=\sum_{l=0}^{d}V_l(y^{h}_{t})\frac{X^{l}_{t_{n},t_{n+1}}}{h},\quad t\in[t_n,t_{n+1}];\\
y^{h}_{0}&=z,
\end{aligned}
\right.
\end{equation}
we define the modified equation for a general method satisfying \eqref{expYn} by the form 
\begin{equation}\label{modified}
\left\{
\begin{aligned}
\dot{\tilde{y}}_{t}
&=\sum_{|\alpha|=1}^{\infty}f_\alpha(\tilde{y}_{t})h^{\alpha_0-1}(X^{1}_{t_{n},t_{n+1}})^{\alpha_1}\cdots(X^{d}_{t_{n},t_{n+1}})^{\alpha_d},\quad t\in[t_n,t_{n+1}];\\
\tilde{y}_{0}&=z,
\end{aligned}
\right.
\end{equation}
where $\tilde{y}_{t}$ is continuous on $[0,T]$. 
Since $|\alpha|\ge 1$, we know that $0\le i(\alpha):=\min\{l:\alpha_l\ge 1,l=0,\cdots,d\}\le d$.
The modified equation can be rewritten in terms of a stochastic equation driven by $x^h$ with the modified vector fields $\bar{V}$:
\begin{align}
d\tilde{y}_t&=\sum_{|\alpha|=1}^{\infty}f_\alpha(\tilde{y}_t)h^{\alpha_0-1}(X^1_{t_n,t_{n+1}})^{\alpha_1}\cdots(X^d_{t_n,t_{n+1}})^{\alpha_d}dt\nonumber\\
&=\sum_{|\alpha|=1}^{\infty}f_\alpha(\tilde{y}_t)h^{\alpha_0}(X^1_{t_n,t_{n+1}})^{\alpha_1}\cdots(X^{i(\alpha)}_{t_n,t_{n+1}})^{\alpha_{i(\alpha)}-1}\cdots(X^d_{t_n,t_{n+1}})^{\alpha_d}\left(\frac{X^{i(\alpha)}_{t_n,t_{n+1}}}{h}\right)dt\nonumber\\
&=\sum_{|\alpha|=1}^{\infty}f_\alpha(\tilde{y}_t)h^{\alpha_0}(X^1_{t_n,t_{n+1}})^{\alpha_1}\cdots(X^{i(\alpha)}_{t_n,t_{n+1}})^{\alpha_{i(\alpha)}-1}\cdots(X^d_{t_n,t_{n+1}})^{\alpha_d}dx^{h,i(\alpha)}_{t}\nonumber\\
&=:\sum_{l=0}^{d}\bar{V}_l(\tilde{y}_t)dx^{h,l}_{t}\nonumber\\
&=:\bar{V}(\tilde{y}_t)dx^{h}_{t}.\label{Wzkmod}
\end{align}
Since the driving signal $x^h$ is of bounded variation, it can be lifted to a $p$-rough path and then the above equation can be interpreted in the rough path sense.

It remains to determine the coefficients $f_\alpha$. 
Using the Taylor expansion and the chain rule, we have
\begin{align}
\tilde{y}_{t_{n+1}}
=&\tilde{y}_{t_{n}}+\sum_{k=1}^{\infty}\frac{d^k}{dt^k}(\tilde{y}_t)\Big\rvert _{t=t_n}\frac{h^k}{k!}\nonumber\\
=&\tilde{y}_{t_{n}}+\sum_{|\alpha|=1}^{\infty}f_\alpha(\tilde{y}_{t_n})h^{\alpha_0}(X^{1}_{t_{n},t_{n+1}})^{\alpha_1}\cdots(X^{d}_{t_{n},t_{n+1}})^{\alpha_d}\nonumber\\
&+\frac{1}{2!}\left[\frac{\partial }{\partial y}\bigg(\sum_{|\alpha|=1}^{\infty}f_\alpha(\tilde{y}_{t_n})h^{\alpha_0}(X^{1}_{t_{n},t_{n+1}})^{\alpha_1}\cdots(X^{d}_{t_{n},t_{n+1}})^{\alpha_d}\bigg)\right]\left(\sum_{|\alpha|=1}^{\infty}f_\alpha(\tilde{y}_{t_n})h^{\alpha_0}(X^{1}_{t_{n},t_{n+1}})^{\alpha_1}\cdots(X^{d}_{t_{n},t_{n+1}})^{\alpha_d}\right)\nonumber\\
&+\frac{1}{3!}\left[\frac{\partial }{\partial y}\bigg(\Big(\frac{\partial }{\partial y}\Big(\sum_{|\alpha|=1}^{\infty}f_\alpha(\tilde{y}_{t_n})h^{\alpha_0}(X^{1}_{t_{n},t_{n+1}})^{\alpha_1}\cdots(X^{d}_{t_{n},t_{n+1}})^{\alpha_d}\Big)\Big)\Big(\sum_{|\alpha|=1}^{\infty}f_\alpha(\tilde{y}_{t_n})h^{\alpha_0}(X^{1}_{t_{n},t_{n+1}})^{\alpha_1}\cdots(X^{d}_{t_{n},t_{n+1}})^{\alpha_d}\Big)\bigg)\right]\nonumber\\
&\qquad\left(\sum_{|\alpha|=1}^{\infty}f_\alpha(\tilde{y}_{t_n})h^{\alpha_0}(X^{1}_{t_{n},t_{n+1}})^{\alpha_1}\cdots(X^{d}_{t_{n},t_{n+1}})^{\alpha_d}\right)+\cdots\nonumber\\
	=:&\tilde{y}_{t_{n}}+\sum_{|\alpha|=1}^{\infty}\tilde{f}_\alpha(\tilde{y}_{t_{n}})h^{\alpha_0}(X^{1}_{t_{n},t_{n+1}})^{\alpha_1}\cdots(X^{d}_{t_{n},t_{n+1}})^{\alpha_d}.\label{expyn}
\end{align}
Introducing the notation $(D_{k^{i_1,i_2}}g)(y):=g'(y)f_{k^{i_1,i_2}}(y)$ and 
\begin{align*}
O_i:=\left\{ (k^{i,1},\cdots,k^{i,i}):  k^{i,1},\cdots,k^{i,i}\in\mathbb{N}^{d+1},~
k^{i,1}_l+\cdots+k^{i,i}_l=\alpha_l,~l=0.\cdots,d
\right\},
\end{align*}
we obtain
\begin{align*}
\tilde{f}_\alpha(y)&=f_\alpha(y),\quad |\alpha|=1,\\
\tilde{f}_\alpha(y)&=f_\alpha(y)+\sum_{i=2}^{|\alpha|}\frac{1}{i!}\sum_{(k^{i,1},\cdots,k^{i,i})\in O_i}(D_{k^{i,1}}\cdots D_{k^{i,i-1}}f_{k^{i,i}})(y),\quad |\alpha|\ge 2.
\end{align*}
 To ensure $Y^h_{n+1}=\tilde{y}_{t_{n+1}}$, comparing \eqref{expYn} and \eqref{expyn} we need
\begin{align*}
\tilde{f}_\alpha(y)=d_\alpha(y),\quad\forall~ \alpha\in\mathbb{N}^{d+1},~|\alpha|\ge 1.
\end{align*}
Therefore, the stochastic modified equation is given by the recursion 
\begin{align}\label{confj}
\begin{split}
f_\alpha(y)&=d_{\alpha}(y),\quad |\alpha|=1,\\
f_\alpha(y)&=d_\alpha(y)-\sum_{i=2}^{|\alpha|}\frac{1}{i!}\sum_{(k^{i,1},\cdots,k^{i,i})\in O_i}(D_{k^{i,1}}\cdots D_{k^{i,i-1}}f_{k^{i,i}})(y),\quad |\alpha|\ge 2.
\end{split}
\end{align}
We note that $f_\alpha$ is determined by the coefficients $d_\alpha$ and $f_{\alpha'}$ with $|\alpha'|<|\alpha|$. 

\subsection{Stochastic modified equation of stochastic symplectic method for stochastic Hamiltonian system}
We consider the stochastic Hamiltonian system in the rough path sense (rough Hamiltonian system for short):
\begin{equation}\label{symeq}
\left\{
\begin{aligned}
dP_t=&-\frac{\partial \mathcal{H}_0(P_t,Q_t)}{\partial Q_t}dt-\sum^{d}_{l=1}\frac{\partial \mathcal{H}_l(P_t,Q_t)}{\partial Q_t}d X_t^{l},\quad P_{0}=p\in \mathbb{R}^{m};\\
dQ_t=&\frac{\partial \mathcal{H}_0(P_t,Q_t)}{\partial P_t}dt+\sum^{d}_{l=1}\frac{\partial \mathcal{H}_l(P_t,Q_t)}{\partial P_t}d X_t^{l},\quad~~~ Q_{0}=q\in \mathbb{R}^{m}.
\end{aligned}
\right.
\end{equation}
One characteristic property of the rough Hamiltonian system is that its phase flow preserves the symplectic structure. More precisely, the differential $2$-form ${\rm d}P \wedge {\rm d}Q$ is invariant under the phase flow. Here the differential is made with respect to the initial value $(p,q)$, which is different from the formal time derivative in \eqref{symeq}. 

\begin{lemma}\label{symp}(\cite[Theorem 3.1]{HHW18})
	The phase flow of the rough Hamiltonian system \eqref{symeq} preserves the symplectic structure, that is,
	\begin{align*}
	{\rm d}P \wedge {\rm d}Q={\rm d}p\wedge {\rm d}q,\quad a.s.
	\end{align*}
\end{lemma}

Denote by $\mathbb{J}_{2m}:=\left( \begin{array}{cc}
0 & \mathbb{I}_{m} \\
-\mathbb{I}_{m} & 0\end{array} \right)$ the standard symplectic matrix. Letting $Y:=(P^\top,Q^\top)^\top $, $z:=(p^\top,q^\top)^\top$ and $V_{l}(y):=\mathbb{J}_{2m}^{-1}\nabla \mathcal{H}_l(y)$, $l=0,\cdots,d$,
we obtain a compact form as equation \eqref{sdeH}.
Thus the stochastic modified equations of  numerical methods satisfying \eqref{expYn} for \eqref{symeq} are constructed similarly as in subsection \ref{sec3.1}. 

Since the symplectic numerical method is implicit in general, the truncation technique with respect to the increments of $X$ are proposed in \cite{Milstein} for the case that $X$ is the standard Brownian motion to avoid the explosion of the moments of the numerical solution. More precisely, the increment $X^l_{t_n,t_{n+1}}$ is substituted by $\Delta_{n+1,l}$, which is defined by 
\begin{align}\label{Delta}
\Delta_{n+1,l}:=\zeta_{n+1,l}\sqrt{h}
\end{align}
with
\begin{equation*}  
\zeta_{n+1,l}:=\left\{  
\begin{array}{ll}  
\xi_{n+1,l}, & |\xi_{n+1,l}|\le A_h,\\
A_h, & \xi_{n+1,l}>A_h,\\
-A_h, & \xi_{n+1,l}<-A_h.
\end{array}  
\right.
\end{equation*}  
Here $\xi_{n+1,l}$, $n=0,1,\cdots,N-1$, $l=1,\cdots,d$,  are independent and identically distributed standard normal random variables, and 
$A_h=\sqrt{k|\ln h|}$ is the threshold with  $k$ large enough such that the convergence order of the numerical method does not decrease. For instance, for numerical methods of strong order $1$, one can choose $k=4$. In this case, the expansion of the numerical solution becomes 
\begin{align}\label{expYn1}
Y^h_{n+1}=Y^h_n+\sum_{|\alpha|=1}^{\infty}d_\alpha(Y^h_n)h^{\alpha_0}\Delta^{\alpha_1}_{n+1,1}\cdots\Delta^{\alpha_d}_{n+1,d},
\end{align}
and then the stochastic modified equation is
\begin{equation}\label{Delmod}
\left\{
\begin{aligned}
\dot{\tilde{y}}_{t}
&=\sum_{|\alpha|=1}^{\infty}f_\alpha(\tilde{y}_{t})h^{\alpha_0-1}\Delta^{\alpha_1}_{n+1,1}\cdots\Delta^{\alpha_d}_{n+1,d}.,\quad t\in[t_n,t_{n+1}];\\
\tilde{y}_{0}&=z,
\end{aligned}
\right.
\end{equation}
where $\tilde{y}_{t}$ is continuous on $[0,T]$ and $f_\alpha$ are defined by \eqref{confj}.

Based on \eqref{Delmod}, we prove in the following that the stochastic modified equation associated to a stochastic symplectic method is still a Hamiltonian system, which gives a positive answer to Problem \ref{P1} in the introduction.

\begin{theorem}\label{tmfj}
	Assume that $V$ is bounded and coutinuously differentiable, and that all its derivatives are bounded. 
	If $Y^h_1(z)$, the one-step numerical solution, is given by applying a symplectic method satisfying \eqref{expYn1} and \eqref{ap-f1} to equation \eqref{symeq}, then for any $f_\alpha:\mathbb{R}^{2m}\rightarrow\mathbb{R}^{2m}$ in \eqref{Delmod}, there exists a Hamiltonian $\mathcal{H}_\alpha:\mathbb{R}^{2m}\rightarrow\mathbb{R}$ such that 
	\begin{align}\label{fjsym}
	f_\alpha(y)=\mathbb{J}_{2m}^{-1}\nabla \mathcal{H}_\alpha(y).
	\end{align}
\end{theorem}
\begin{proof}
	From condition \eqref{ap-f1}, we have immediately that \eqref{fjsym} holds for $|\alpha|=1$. 
	Define $\theta(\alpha):=\alpha_0+\frac{\alpha_1+\cdots+\alpha_d}{2}$. Notice that if $|\alpha'|=|\alpha|+1$, then $\theta(\alpha')\ge \theta(\alpha)+\frac12$. For $r\in\mathbb{N}_+$, assume by induction that for any $\alpha$ such that $\theta(\alpha)\le \frac{r}{2}$, \eqref{fjsym} holds. Consider the truncated modified equation as follows:
	\begin{align*}
	\dot{\tilde{y}}^{r}_{t}&=\sum_{\theta(\alpha)=1/2}^{r/2}f_\alpha(\tilde{y}_{t}^{r})h^{\alpha_0-1}\Delta^{\alpha_1}_{1,1}\cdots\Delta^{\alpha_d}_{1,d},\quad \tilde{y}_{0}^{r}=z.
	\end{align*}
	Denote by $\pi^r(z)_t$ the flow of this truncated modified equation. Together with the assumptions, the recursion \eqref{confj} implies that there exists a random variable $R_{\theta(\alpha)}$ such that
	$\|R_{\theta(\alpha)}\|_{L^2(\Omega)}\le Ch^{\frac{r+2}{2}}$ and
	\begin{align*}
	Y^h_1(z)&=\pi^r(z)_h+\sum_{\theta(\alpha)=(r+1)/2}f_\alpha(z)h^{\alpha_0}\Delta^{\alpha_1}_{1,1}\cdots\Delta^{\alpha_d}_{1,d}+R_{\theta(\alpha)}, \quad a.s.
	\end{align*}
	Moreover, the Jacobian of the flow satisfies that 
	\begin{align*}
	\frac{\partial Y^h_1(z)}{\partial z}&=\frac{\partial \pi^r(z)_h}{\partial z}+\sum_{\theta(\alpha)=(r+1)/2}f'_\alpha(z)h^{\alpha_0}\Delta^{\alpha_1}_{1,1}\cdots\Delta^{\alpha_d}_{1,d}+
	\frac{\partial R_{\theta(\alpha)}}{\partial z}, \quad a.s.,\\
	\frac{\partial \pi^r(z)_h}{\partial z}&=\mathbb{I}_{2m}+R_h,\quad a.s.,
	\end{align*}
	with $f'_\alpha(z)=\frac{\partial f_\alpha(z)}{\partial z}$, $\Big\|\frac{\partial R_{\theta(\alpha)}}{\partial z}\Big\|_{L^p(\Omega)}\le C(p)h^{\frac{r+2}{2}}$, $\|R_h\|_{L^p(\Omega)}\le C(p)h^{\frac12}$ and $p\ge 2$.
	Since the method is symplectic, we have
	\begin{align*}
	\mathbb{J}_{2m}=\Big(\frac{\partial Y^h_1(z)}{\partial z}\Big)^\top \mathbb{J}_{2m} ~\frac{\partial Y^h_1(z)}{\partial z},\quad a.s.
	\end{align*}
    Substituting the expressions of the Jacobian $\frac{\partial Y^h_1(z)}{\partial z}$ into the above equality, we obtain 
	\begin{align*}
	\mathbb{J}_{2m}=&\Big(\frac{\partial \pi^r(z)_h}{\partial z}\Big)^\top \mathbb{J}_{2m} ~\frac{\partial \pi^r(z)_h}{\partial z}
	 +\sum_{\theta(\alpha)=(r+1)/2}\mathbb{J}_{2m} ~ f'_\alpha(z) h^{\alpha_0}\Delta^{\alpha_1}_{1,1}\cdots\Delta^{\alpha_d}_{1,d}
	 \\
	&+ \sum_{\theta(\alpha)=(r+1)/2}f'_\alpha(z)^\top 
	\mathbb{J}_{2m} ~  h^{\alpha_0}\Delta^{\alpha_1}_{1,1}\cdots\Delta^{\alpha_d}_{1,d}
	+R, \quad a.s.,
	\end{align*}
	where $\|R\|_{L^p(\Omega)}\le C(p)h^{\frac{r+2}{2}}$.
	Together with the induction assumption and the definition \eqref{Delta} of $\Delta_{1,l}$, it holds that 
	\begin{align*}
	\left[\sum_{\theta(\alpha)=(r+1)/2}\left(\mathbb{J}_{2m} ~ f'_\alpha(z)
+ f'_\alpha(z)^\top 
	\mathbb{J}_{2m} ~\right)  \zeta_{1,1}^{\alpha_1}\cdots \zeta_{1,d}^{\alpha_d}\right]
	=-h^{-\frac{r+1}{2}} R, \quad a.s. 
	\end{align*}
  Noticing that $\zeta_{1,l}$, $l=1,\cdots,d$ are truncated normal distribution random variables, we deduce that the left side of the above equality converges to 
  \begin{align*}
  \sum_{\theta(\alpha)=(r+1)/2}\left(\mathbb{J}_{2m} ~ f'_\alpha(z)
  + f'_\alpha(z)^\top 
  \mathbb{J}_{2m} ~\right)  \xi_{1,1}^{\alpha_1}\cdots \xi_{1,d}^{\alpha_d}
  \end{align*}
in $L^2(\Omega)$ sense as $h$ goes to $0$. On the other hand, the fact that $\big\|h^{-\frac{r+1}{2}} R\big\|_{L^2(\Omega)}$ converges to $0$ leads to 
	\begin{align}\label{J}
	\sum_{\theta(\alpha)=(r+1)/2}\left(\mathbb{J}_{2m} ~ f'_\alpha(z) 
	+ f'_\alpha(z)^\top 
	\mathbb{J}_{2m} \right)  \xi_{1,1}^{\alpha_1}\cdots \xi_{1,d}^{\alpha_d}=0,\quad a.s.
	\end{align}
	In \cite[Theorem 1.6]{Orth}, it is proved that there exists a unique monic orthogonal polynomial sequence $\{p_k(x)\}^\infty_{k=1}$ with respect to the measure induced by $\xi_{1,1}$, i.e., the Hermite polynomials. We rewrite \eqref{J} as
	\begin{align*}
	0=&\sum_{\theta(\alpha)=(r+1)/2,(\alpha_1+\cdots+\alpha_d)/2=\theta(\alpha)}\left(\mathbb{J}_{2m} ~ f'_\alpha(z)
	+ f'_\alpha(z)^\top 
	\mathbb{J}_{2m} \right)  p_{\alpha_1}(\xi_{1,1})\cdots p_{\alpha_d}(\xi_{1,d})\\
	&+\bigg[-\sum_{\theta(\alpha)=(r+1)/2,(\alpha_1+\cdots+\alpha_d)/2=\theta(\alpha)}\left(\mathbb{J}_{2m} ~ f'_\alpha(z) 
	+ f'_\alpha(z)^\top 
	\mathbb{J}_{2m} \right) \Big(p_{\alpha_1}(\xi_{1,1})\cdots p_{\alpha_d}(\xi_{1,d})- \xi_{1,1}^{\alpha_1}\cdots \xi_{1,d}^{\alpha_d}\Big)\\
	&+\sum_{\theta(\alpha)=(r+1)/2,(\alpha_1+\cdots+\alpha_d)/2<\theta(\alpha)}\left(\mathbb{J}_{2m} ~ f'_\alpha(z)
	+ f'_\alpha(z)^\top 
	\mathbb{J}_{2m} \right)  \xi_{1,1}^{\alpha_1}\cdots \xi_{1,d}^{\alpha_d}\bigg]\\
	=&:\sum_{\theta(\alpha)=(r+1)/2,(\alpha_1+\cdots+\alpha_d)/2=\theta(\alpha)}\left(\mathbb{J}_{2m} ~ f'_\alpha(z)
	+ f'_\alpha(z)^\top 
	\mathbb{J}_{2m} \right)  p_{\alpha_1}(\xi_{1,1})\cdots p_{\alpha_d}(\xi_{1,d})\\
	&+\sum_{(\alpha_1+\cdots+\alpha_d)/2<\theta(\alpha)}c_{\alpha_1,\cdots,\alpha_d} \xi_{1,1}^{\alpha_1}\cdots \xi_{1,d}^{\alpha_d},\quad a.s.,
	\end{align*}
	where we have used the fact that $p_{\alpha_1},\cdots,p_{\alpha_d}$ are monic. Similar arguments lead to 
	
	\begin{align*}
	0=&\sum_{\theta(\alpha)=(r+1)/2,(\alpha_1+\cdots+\alpha_d)/2=\theta(\alpha)}\left(\mathbb{J}_{2m} ~ f'_\alpha(z)
	+ f'_\alpha(z)^\top 
	\mathbb{J}_{2m} \right)  p_{\alpha_1}(\xi_{1,1})\cdots p_{\alpha_d}(\xi_{1,d})\\
	&+\sum_{(\alpha_1+\cdots+\alpha_d)/2<\theta(\alpha)}\bar{c}_{\alpha_1,\cdots,\alpha_d}p_{\alpha_1}(\xi_{1,1})\cdots p_{\alpha_d}(\xi_{1,d}),\quad a.s.
	\end{align*}
	For any $\alpha$ satisfying $\theta(\alpha)=(r+1)/2$ and $(\alpha_1+\cdots+\alpha_d)/2=\theta(\alpha)$, multiplying the above equation by $p_{\alpha_1}(\xi_{1,1})\cdots p_{\alpha_d}(\xi_{1,d})$ and taking the expectation, we deduce from the independency of $\xi_{1,1},\cdots,\xi_{1,d}$ and the orthogonality of $\{p_k(x)\}^\infty_{k=1}$ that 
	\begin{align*}
	\mathbb{J}_{2m} ~ f'_\alpha(z)
	+ f'_\alpha(z)^\top 
	\mathbb{J}_{2m}=0.
	\end{align*}
	Plugging it into \eqref{J} and rewriting it as before, we have
	\begin{align*}
	0=&\sum_{\theta(\alpha)=(r+1)/2,(\alpha_1+\cdots+\alpha_d)/2=\theta(\alpha)-1}\left(\mathbb{J}_{2m} ~ f'_\alpha(z) 
	+ f'_\alpha(z)^\top 
	\mathbb{J}_{2m} ~\right)  p_{\alpha_1}(\xi_{1,1})\cdots p_{\alpha_d}(\xi_{1,d})\\
	&+\sum_{(\alpha_1+\cdots+\alpha_d)/2<\theta(\alpha)-1}\tilde{c}_{\alpha_1,\cdots,\alpha_d} p_{\alpha_1}(\xi_{1,1})\cdots p_{\alpha_d}(\xi_{1,d}),\quad a.s.
	\end{align*}
	Then we also have for any $\alpha$ satisfying $\theta(\alpha)=(r+1)/2$ and $(\alpha_1+\cdots+\alpha_d)/2=\theta(\alpha)-1$, $\mathbb{J}_{2m} ~ f'_\alpha(z)+ f'_\alpha(z)^\top 
	\mathbb{J}_{2m} =0$. Repeatly using previous arguments, we finally have for any $\alpha$ satisfying $\theta(\alpha)=(r+1)/2$, $\mathbb{J}_{2m} ~ f'_\alpha(z)+ f'_\alpha(z)^\top 
	\mathbb{J}_{2m} =0$, i.e., $\mathbb{J}_{2m} ~ f'_\alpha(z)$ is symmetric. Then the statement \eqref{fjsym} follows from the integrability lemma \cite[Lemma 2.7 in Chap.  \uppercase\expandafter{\romannumeral6}]{GeometricEH}.
\end{proof}

\begin{remark}\label{rm-TRM}
In the proof of Theorem \ref{tmfj}, the truncated increments $\Delta_{1,l}$, $l=1,\cdots,d$ are used to ensure the $L^p(\Omega)$-integrability of the remainders. In fact, the coefficients $f_j$  are independent of the values of $\Delta_{1,l}$. Therefore, a similar result holds for the method \eqref{expYn} containing the powers of $X^l_{t_n,t_{n+1}}$. Besides, for the weak convergent  symplectic method which approximates $X^l_{t_n,t_{n+1}}$ by $\varsigma_{ln}\sqrt{h}$ with the random variable $\varsigma_{ln}$ defined through $\mathbb{P}(\varsigma_{ln}=\pm)=\frac12$, such as the method studied in \cite{2019BIT},  one can constructe the modified equation by regarding $X^l_{t_n,t_{n+1}}$ as  $\varsigma_{ln}\sqrt{h}$ and get the symplecticity of the modified equation.
\end{remark}

Based on \eqref{fjsym} and Remark \ref{rm-TRM}, we state a more general result in the following.

\begin{theorem}\label{tmsym}
	Assume that $V$ is bounded and coutinuously differentiable, and that all its derivatives are bounded. 
	If $Y^h_1(z)$, the one-step numerical solution, is given by applying a symplectic method satisfying  \eqref{expYn}-\eqref{ap-f1} to equation \eqref{symeq}, then the associated stochastic modified equation  \eqref{modified} is a Hamiltonian system. 
\end{theorem}

\section{Convergence analysis}\label{sec4}

In general, the stochastic modified equation is a formal one, i.e., the series given in \eqref{modified} may not converge. We consider the $\tilde{N}$-truncated modified equation ($\tilde{N}\ge 1$)
\begin{equation}\label{TME}
\left\{
\begin{aligned}
\dot{\tilde{y}}^{\tilde{N}}_{t}&=\sum_{|\alpha|=1}^{\tilde{N}}f_\alpha(\tilde{y}_{t}^{\tilde{N}})h^{\alpha_0-1}(X^{1}_{t_{n},t_{n+1}})^{\alpha_1}\cdots(X^{d}_{t_{n},t_{n+1}})^{\alpha_d},\quad t\in[t_n,t_{n+1}];\\ \tilde{y}^{\tilde{N}}_{0}&=z.
\end{aligned}
\right.
\end{equation}
In subsection \ref{sub4.1}, we give the convergence analysis on the error between $Y^h_n$ and  $\tilde{y}^{\tilde{N}}_{t_n}$ for the case that $X$ is a general Gaussian rough path satisfying Assumption \ref{R}, which answers Problem \ref{P2}. As for Problem \ref{P3}, we focus on the case that $X$ is the standard Brownian motion and the numerical method is obtained via the truncated increments in \eqref{Delta}. We provide an approach to optimizing $\tilde{N}$ such that the error is exponentially small with respect to $h$, in subsection \ref{sub4.2}.

\subsection{The general rough case}\label{sub4.1}
 
\begin{theorem}\label{tm-Nlocal}
Under Assumption \ref{R},
if $V\in Lip^{\tilde{N}+1}$, 
then for any $p>2\rho$, there exists a random variable $C(\omega)=C(\omega,p,\|V\|_{Lip^{\tilde{N}+1}},\tilde{N})$ such that 
\begin{align*}
\|\tilde{y}^{\tilde{N}}_{t_1}-Y^{h}_{1}\|\le C(\omega)h^{\frac{\tilde{N}+1}{p}}, \quad{\rm a.s.},
\end{align*}
where $\tilde{y}^{\tilde{N}}$ is the solution of \eqref{TME} and $Y^h_1$ is defined by a numerical method satisfying \eqref{expYn}.
\end{theorem}
\begin{proof}
	Consider the expansion
	\begin{align*}
	\tilde{y}^{\tilde{N}}_h=z+\sum_{|\alpha|=1}^{\infty}f^{\tilde{N}}_\alpha(z)h^{\alpha_0}(X^{1}_{t_{0},t_{1}})^{\alpha_1}\cdots(X^{d}_{t_{0},t_{1}})^{\alpha_d}.
	\end{align*}
	Fix $p>2\rho\ge 2$.
     Since the recursion \eqref{confj} implies
	$f^{\tilde{N}}_\alpha=f_\alpha=d_\alpha$ with $1\le |\alpha|\le \tilde{N}$, and Assumption \ref{R} produces $\|X\|_{\frac1p\text{-}{\rm H\ddot{o}l};[t_0,t_1]}<\infty$, 
	we deduce from the Taylor's expansion that the leading term of the error between $\tilde{y}^{\tilde{N}}_{t_1}$ and $Y^{h}_{1}$ is involved with $h^{\alpha_0}(X^{1}_{t_{0},t_{1}})^{\alpha_1}\cdots(X^{d}_{t_{0},t_{1}})^{\alpha_d}$, where $\alpha_0=0$ and $\alpha_1\cdots+\alpha_d=\tilde{N}+1$.
	Hence,
\begin{align*}
\|\tilde{y}^{\tilde{N}}_{t_1}-Y^{h}_{1}\|\le C(\omega,p,\|V\|_{Lip^{\tilde{N}+1}},\tilde{N})h^{\frac{\tilde{N}+1}{p}}.
\end{align*}
\end{proof}

\begin{theorem}\label{tm-Nglobal}
	Under Assumption \ref{R},
	if $V\in Lip^{\tilde{N}+\gamma}$ with $\gamma>2\rho$ and $\tilde{N}>2\rho-1$,
	then for any $\gamma>p>2\rho$, there exists a random variable $C(\omega)=C(\omega,p,\gamma,\|V\|_{Lip^{\tilde{N}+\gamma}},\tilde{N},T)$  such that 
\begin{align*}
\sup_{1\le n\le N}\|\tilde{y}^{\tilde{N}}_{t_n}-Y^{h}_{n}\|\le C(\omega) h^{\frac{\tilde{N}+1}{p}-1}, \quad{\rm a.s.},
\end{align*}
where $\tilde{y}^{\tilde{N}}$ is the solution of \eqref{TME} and $Y^h_n$ is defined by a numerical method satisfying \eqref{expYn}.
\end{theorem}

\begin{proof}
	Similar to \eqref{Wzkmod}, we rewrite the truncated modified equation \eqref{TME} into
   \begin{align*}
d\tilde{y}^{\tilde{N}}_t&=\sum_{|\alpha|=1}^{\tilde{N}}f_\alpha(\tilde{y}^{\tilde{N}}_t)h^{\alpha_0-1}(X^1_{t_n,t_{n+1}})^{\alpha_1}\cdots(X^d_{t_n,t_{n+1}})^{\alpha_d}dt\\
&=\sum_{|\alpha|=1}^{\tilde{N}}f_\alpha(\tilde{y}^{\tilde{N}}_t)h^{\alpha_0}(X^1_{t_n,t_{n+1}})^{\alpha_1}\cdots(X^{i(\alpha)}_{t_n,t_{n+1}})^{\alpha_{i(\alpha)}-1}\cdots(X^d_{t_n,t_{n+1}})^{\alpha_d}\left(\frac{X^{i(\alpha)}_{t_n,t_{n+1}}}{h}\right)dt\\
&=\sum_{|\alpha|=1}^{\tilde{N}}f_\alpha(\tilde{y}^{\tilde{N}}_t)h^{\alpha_0}(X^1_{t_n,t_{n+1}})^{\alpha_1}\cdots(X^{i(\alpha)}_{t_n,t_{n+1}})^{\alpha_{i(\alpha)}-1}\cdots(X^d_{t_n,t_{n+1}})^{\alpha_d}dx^{h,i(\alpha)}_{t}\\
&=:\bar{V}^{\tilde{N}}(\tilde{y}^{\tilde{N}}_t)dx^{h}_{t},\quad t\in(t_n,t_{n+1}].
   \end{align*}
	This shows that \eqref{TME} is equivalent to a rough differential equation with vector field $\bar{V}^{\tilde{N}}$ driven by $x^h_t$. Denoting by  $\pi(t_{0},y_{0},x^{h})_{t}$, $t \geq t_{0}$ its flow with the initial value $y_{0}$ at time $t_{0}$, we have
	\begin{align*}
	\|Y^{h}_{k}-\tilde{y}^{\tilde{N}}_{t_k}\|
	=&\|\pi(t_{k},Y^{h}_{k},x^{h})_{t_{k}}-\pi(t_{0},Y^{h}_{0},x^{h})_{t_{k}}\|\\
	\le& \sum^{k}_{s=1}\|\pi(t_{s},Y^{h}_{s},x^{h})_{t_{k}}-\pi(t_{s-1},Y^{h}_{s-1},x^{h})_{t_{k}}\|,\quad 1\le k\le N.
	\end{align*}
	According to the Lipschitz continuity of the It\^o--Lyons map (see e.g. \cite[Theorem 10.26]{Friz}), we get
	\begin{align*}
	&\|\pi(t_{s},Y^{h}_{s},x^{h})_{t_{k}}-\pi(t_{s-1},Y^{h}_{s-1},x^{h})_{t_{k}}\|\\
	=&\|\pi(t_{k-1},\pi(t_{s},Y^{h}_{s},x^{h})_{t_{k-1}},x^{h})_{t_{k}}-\pi(t_{k-1},\pi(t_{s-1},Y^{h}_{s-1},x^{h})_{t_{k-1}},x^{h})_{t_{k}}\|\\
	\leq& C\exp\{C\bar{\nu}^{p}\|S_{[p]}(x^{h})(\omega)\|^{p}_{p\text{-}var;[t_{k-1},t_{k}]}\}\|\pi(t_{s},Y^{h}_{s},x^{h})_{t_{k-1}}-\pi(t_{s-1},Y^{h}_{s-1},x^{h})_{t_{k-1}}\|,\quad 1\le s<k,
	\end{align*}
	where $C=C(p,\gamma)$ and $\bar{\nu}=\bar{\nu}(\|X\|_{\frac1p\text{-}{\rm H\ddot{o}l};[0,T]}(\omega),\|V\|_{Lip^{\tilde{N}+\gamma}},\tilde{N})\ge\|\bar{V}\|_{Lip^{\gamma}}$.
	From
	\begin{align*}
	\|S_{[p]}(x^{h})(\omega)\|^{p}_{p\text{-}var;[u_1,u_2]}+\|S_{[p]}(x^{h})(\omega)\|^{p}_{p\text{-}var;[u_2,u_3]}\le \|S_{[p]}(x^{h})(\omega)\|^{p}_{p\text{-}var;[u_1,u_3]}, \quad 0\le u_1<u_2<u_3\le T,
	\end{align*}
	it yields that
	\begin{align*}
	&\|\pi(t_{s},Y^{h}_{s},x^{h})_{t_{k}}-\pi(t_{s-1},Y^{h}_{s-1},x^{h})_{t_{k}}\|\\
	\leq& C\exp\{C\bar{\nu}^{p}\|S_{[p]}(x^{h})(\omega)\|^{p}_{p\text{-}var;[t_{s},t_{k}]}\}\|\pi(t_{s},Y^{h}_{s},x^{h})_{t_{s}}-\pi(t_{s-1},Y^{h}_{s-1},x^{h})_{t_{s}}\|\\
	\leq& C\exp\{C\bar{\nu}^{p}\|S_{[p]}(x^{h})(\omega)\|^{p}_{p\text{-}var;[0,T]}\}\|Y^{h}_s-\pi(t_{s-1},Y^{h}_{s-1},x^{h})_{t_{s}}\|,\quad 1\le s\le k.
	\end{align*}
	Recall that for almost all $\omega\in\Omega$, $\|S_{[p]}(x^{h})(\omega)\|^{p}_{p\text{-}var;[0,T]}$ is uniformly bounded with respect to $h$ (see \cite[Theorem 15.28]{Friz}).
	We derive by Theorem \ref{tm-Nlocal} that 
	\begin{align*}
	\|Y^{h}_{k}-\tilde{y}^{\tilde{N}}_{t_k}\|
	\le& \sum^{k}_{s=1}C\exp\{C\bar{\nu}^{p}\|S_{[p]}(x^{h})(\omega)\|^{p}_{p\text{-}var;[0,T]}\}\|Y^{h}_s-\pi(t_{s-1},Y^{h}_{s-1},x^{h})_{t_{s}}\|\\
	\le& C(\omega,p,\gamma,\|V\|_{Lip^{\tilde{N}+\gamma}},\tilde{N},T)h^{\frac{\tilde{N}+1}{p}-1},
	\end{align*}
	due to the fact $\gamma>2\rho\ge2$.
\end{proof}

In case of additive noise, since the diffusion part can be simulated exactly, the assumptions \eqref{expYn}-\eqref{ap-f1} on the numerical method degenerate to 
 \begin{align}\label{expYnadd}
Y^h_{n+1}=Y^h_n+\sum_{|\alpha|=1}V_\alpha(Y^h_n)h^{\alpha_0}(X^{1}_{t_{n},t_{n+1}})^{\alpha_1}\cdots(X^{d}_{t_{n},t_{n+1}})^{\alpha_d}+\sum_{|\alpha|=2,\alpha_0\ge 1}^{\infty}d_\alpha(Y^h_n)h^{\alpha_0}(X^{1}_{t_{n},t_{n+1}})^{\alpha_1}\cdots(X^{d}_{t_{n},t_{n+1}})^{\alpha_d}.
\end{align}
Consequently, the convergence rate of the error between $Y^h_n$ and $\tilde{y}^{\tilde{N}}_{t_n}$  is improved, which is stated in the following corollary.

\begin{corollary}\label{tm-Nglobaladd}
		Let Assumption \ref{R} hold and $V_i(y)\equiv\sigma_i\in \mathbb{R}^m$, $i=1,\cdots,d$. 
	If $V_0\in Lip^{\tilde{N}+\gamma}$ with $\gamma>2\rho$,
	then for any $\gamma>p>2\rho$, there exists a random variable $C(\omega)=C(\omega,p,\gamma,\|V_0\|_{Lip^{\tilde{N}+\gamma}},\sigma_i,\tilde{N},T)$  such that 
	\begin{align*}
\sup_{0\le n\le N}\|\tilde{y}^{\tilde{N}}_{t_n}-Y^{h}_{n}\|\le C(\omega) h^{\frac{\tilde{N}}{p}}, \quad{\rm a.s.},
\end{align*}
	where $\tilde{y}^{\tilde{N}}$ is the solution of \eqref{TME} and $Y^h_n$ is defined by a numerical method satisfying \eqref{expYnadd}.
\end{corollary}
\begin{proof}
Combining \eqref{expYnadd} with \eqref{confj}, we have that 
the leading term of the local error between $\tilde{y}^{\tilde{N}}_{t_1}$ and $Y^{h}_{1}$ is involved with $h^{\alpha_0}(X^{1}_{t_{0},t_{1}})^{\alpha_1}\cdots(X^{d}_{t_{0},t_{1}})^{\alpha_d}$, where $\alpha_0=1$ and $\alpha_1\cdots+\alpha_d=\tilde{N}$. Then 
\begin{align*}
\|\tilde{y}^{\tilde{N}}_{t_1}-Y^{h}_{1}\|\le C(\omega)h^{\frac{\tilde{N}}{p}+1}, \quad{\rm a.s.},
\end{align*}
from which we conclude the result by using the same arguments as in the proof of Theorem \ref{tm-Nglobal}.
\end{proof}


\subsection{The standard Brownian case}\label{sub4.2}
In this subsection, we assume that $X^l$, $l=1,\cdots,d$ are independent standard Brownian motions. 
For convenience, we illustrate our idea by the RK method
\begin{equation}\label{TrunRK}
\left\{\
\begin{aligned}
Y^h_{n+1,i}&=Y^h_{n}+\sum^{s}_{j=1}a_{ij}\left(V_0(Y^h_{n+1,j})h+\sum_{l=1}^{d}V_l(Y^h_{n+1,j})\Delta_{n+1,l}\right),\\
Y^h_{n+1}&=Y^h_{n}+\sum^{s}_{i=1}b_{i}\left(V_0(Y^h_{n+1,i})h+\sum_{l=1}^{d}V_l(Y^h_{n+1,i})\Delta_{n+1,l}\right),
\end{aligned}
\right.
\end{equation}
where $\Delta_{n+1,l}$ is defined in \eqref{Delta}. We also stress that the procedure does not rely on the special structure of RK methods and is avaliable for a large class of numerical methods.

We first show that the method containing $\Delta_{n+1,l}$ also fits into the previous convergence analysis when the stochatic modified equation is \eqref{Delmod}. 
Then it sufficies to prove that the process $\bar{x}^{h}=(\bar{x}^{h,1},\cdots,\bar{x}^{h,d})$, which is defined by
\begin{align*}
\bar{x}^{h,l}_t:= \bar{x}^{h,l}_{t_n}+\dfrac{t-t_{n}}{h}\Delta_{n+1,l},\quad \forall~t \in (t_{n},t_{n+1}],~~l=1,\cdots,d,~~ n=0,\cdots,N-1,
\end{align*}
 can be lifted to a $p$-rough path with $[p]=2$ almost surely. 

\begin{proposition}
	Let $2<p<3$. Then it holds that there exists some random variable $C(\omega):=C(\omega,p,T)$ independent of $h$ such that 
	\begin{align*}
	\left\|  S_2(\bar{x}^h(\omega)) \right\|_{p\text{-}var;[0,T]}\le C(\omega),\quad a.s.
	\end{align*}
\end{proposition}
\begin{proof}
	Let $t_{i-1}<s<t_i<t_j<t<t_{j+1}$. Since for any $m\in \mathbb{N}_+$, $\mathbb{E}\left[ \Delta_{1,1} ^{2m}\right]\le (2m-1)!!h^m$ and $\mathbb{E}\left[\Delta_{1,1} ^{2m-1}\right]=0$, we have
	\begin{align*}
	\mathbb{E}\left[\left| \int_{t_i}^{t_j}d\bar{x}^{h,l}_{u_1} \right|^{2m}\right]
	&=\mathbb{E}\left[\left(  \sum_{k=i+1}^{j}  \Delta_{k,l} \right)^{2m}\right]\le (2m)! (j-i)^m \mathbb{E}\left[\left|   \Delta_{1,1} \right|^{2m}\right]\le C(m)|t_j-t_i|^m,\\
	\mathbb{E}\left[\left| \int_{s}^{t_i}d\bar{x}^{h,l}_{u_1} \right|^{2m}\right]
	&=\mathbb{E}\left[\left( \frac{t_i-s}{h} \Delta_{i,l} \right)^{2m}\right]\le  \left(\frac{t_i-s}{h} \right)^{2m}\mathbb{E}\left[\left|   \Delta_{1,1} \right|^{2m}\right]
	\le  (2m-1)!! \frac{(t_i-s)^{2m}}{h^{m}} 
	\le C(m)|t_i-s|^m,\\
		\mathbb{E}\left[\left| \int_{t_j}^{t}d\bar{x}^{h,l}_{u_1} \right|^{2m}\right]
	&=\mathbb{E}\left[\left( \frac{t-t_j}{h} \Delta_{i,l} \right)^{2m}\right]\le 
	 \left(\frac{t-t_j}{h} \right)^{2m}\mathbb{E}\left[\left|   \Delta_{1,1} \right|^{2m}\right]
	 \le  (2m-1)!! \frac{(t-t_j)^{2m}}{h^{m}} 
	 \le C(m)|t-t_j|^m.
	\end{align*}
	Combining the above estimates, we obtain
	\begin{align*}
	\mathbb{E}\left[\left| \int_{s}^{t}d\bar{x}^{h,l}_{u_1} \right|^{2m}\right]&\le C(m)\left\{\mathbb{E}\left[\left| \int_{s}^{t_i}d\bar{x}^{h,l}_{u_1} \right|^{2m}\right]+\mathbb{E}\left[\left| \int_{t_i}^{t_j}d\bar{x}^{h,l}_{u_1} \right|^{2m}\right]+\mathbb{E}\left[\left| \int_{t_j}^{t}d\bar{x}^{h,l}_{u_1} \right|^{2m}\right]\right\}\\
	&\le C(m)|t-s|^m.
	\end{align*}

	For iterated integral, letting $t_{i-1}<s<t_i<t_j<t<t_{j+1}$ and $l_1,l_2\in\{1,\cdots,d\}$, we derive
	\begin{align*}
	\mathbb{E}\left[\left| \int_{t_i}^{t_j}\int_{t_i}^{u_1}d\bar{x}^{h,l_1}_{u_2}d\bar{x}^{h,l_2}_{u_1} \right|^{2m}\right]
	&\le (4m)! (j-i)^{2m} \mathbb{E}\left[\left|   \Delta_{1,1} \right|^{4m}\right]\le C(m)|t_j-t_i|^{2m},\quad l_1=l_2,\\
	\mathbb{E}\left[\left| \int_{t_i}^{t_j}\int_{t_i}^{u_1}d\bar{x}^{h,l_1}_{u_2}d\bar{x}^{h,l_2}_{u_1} \right|^{2m}\right]
	&\le \left((2m)!(j-i)^{m} \mathbb{E}\left[\left|   \Delta_{1,1} \right|^{2m}\right]\right)^2\le C(m)|t_j-t_i|^{2m},\quad l_1\neq l_2.
	\end{align*}
	Besides, 
   	\begin{align*}
   \mathbb{E}\left[\left| \int_{t_i}^{t_j}\int_{s}^{t_i}d\bar{x}^{h,l_1}_{u_2}d\bar{x}^{h,l_2}_{u_1} \right|^{2m}\right]
   &\le \left(\frac{t_i-s}{h} \right)^{2m}\mathbb{E}\left[\left|   \Delta_{1,1} \right|^{2m} \right] \mathbb{E}\left[\left| \int_{t_i}^{t_j}d\bar{x}^{h,l}_{u_1} \right|^{2m}\right]
   \le C(m)|t_i-s|^m|t_j-t_i|^m,\\
   \mathbb{E}\left[\left| \int_{s}^{t_i}\int_{s}^{u_1}d\bar{x}^{h,l_1}_{u_2}d\bar{x}^{h,l_2}_{u_1} \right|^{2m}\right]
   &\le \left(\frac{t_i-s}{h} \right)^{4m}\left(\mathbb{E}\left[\left|   \Delta_{1,1} \right|^{2m}\right]\right)^2
   \le C(m)|t_i-s|^{2m}.
   \end{align*}
	Similarly, it holds that
	\begin{align*}
	\mathbb{E}\left[\left| \int_{t_j}^{t}\int_{s}^{u_1}d\bar{x}^{h,l_1}_{u_2}d\bar{x}^{h,l_2}_{u_1} \right|^{2m}\right]
\le C(m)|t-t_j|^m|t_j-s|^m+C(m)|t-t_j|^{2m}.
	\end{align*}
	Therefore, we obtain
	\begin{align*}
	\mathbb{E}\left[\left| \int_{s}^{t}\int_{s}^{u_1}d\bar{x}^{h,l_1}_{u_2}d\bar{x}^{h,l_2}_{u_1} \right|^{2m}\right]\le C(m)|t-s|^{2m}.
	\end{align*}
	
	For any $p$ such that $2<p<3$, i.e., $\frac13<\frac1p<\frac12$, choose $m\in\mathbb{N}_+$ such that $q:=4m>\frac{4p}{p-2}$ which implies $(\frac{1}{2}-\frac{1}{q})-\frac{1}{q}>\frac{1}{p}$. 
    By the Besov-H\"older embedding theorem \cite[Corollary A.2]{Friz}, we get
    \begin{align*}
    \left\|  S_2(\bar{x}^h) \right\|^q_{((\frac12-\frac1q)-\frac1q)\text{-}{\rm H\ddot{o}l};[0,T]}
    \le C(q)
    \int_{0}^{T}\int_{0}^{T}\frac{\left|  {\rm d} (S_2(\bar{x}^h)_s,S_2(\bar{x}^h)_t)\right|^q}{|t-s|^{1+q(\frac12-\frac1q)}}dsdt,
    \end{align*}
    where
		\begin{align*}
	{\rm d} (S_2(\bar{x}^h)_s,S_2(\bar{x}^h)_t)\le C\max\left\{ \left| \int_{s}^{t}d\bar{x}^{h,l}_{u_1} \right|,
	\left| \int_{s}^{t}\int_{s}^{u_1}d\bar{x}^{h,l_1}_{u_2}d\bar{x}^{h,l_2}_{u_1} \right|^{\frac12} 
	\right\}\\
	\le C\left(\left| \int_{s}^{t}d\bar{x}^{h,l}_{u_1} \right|+
	\left| \int_{s}^{t}\int_{s}^{u_1}d\bar{x}^{h,l_1}_{u_2}d\bar{x}^{h,l_2}_{u_1} \right|^{\frac12} 
	\right).
	\end{align*}
	Taking the expectation on both sides, we obtain 
	\begin{align*}
	\mathbb{E}\left[\left\|  S_2(\bar{x}^h) \right\|^q_{(\frac12-\frac1q-\frac1q)\text{-}{\rm H\ddot{o}l};[0,T]}\right]
	&\le C(q)
	\int_{0}^{T}\int_{0}^{T}\frac{\mathbb{E}\left[\left|  {\rm d} (S_2(\bar{x}^h)_s,S_2(\bar{x}^h)_t)\right|^q\right]}{|t-s|^{\frac{q}{2}}}dsdt\\
	&\le C(q)
	\int_{0}^{T}\int_{0}^{T}\frac{\mathbb{E}\left[\left| \int_{s}^{t}d\bar{x}^{h,l}_{u_1} \right|^q+
		\left| \int_{s}^{t}\int_{s}^{u_1}d\bar{x}^{h,l_1}_{u_2}d\bar{x}^{h,l_2}_{u_1} \right|^{\frac{q}{2}} \right]}{|t-s|^{\frac{q}{2}}}dsdt\\
	&\le C(q,m)T^2.
	\end{align*}
	This yields that $\bar{x}^h$ can be lifted to a $p$-rough path almost surely, and that there exists some random variable $C(\omega)$ independent of $h$ such that 
	\begin{align*}
	\left\|  S_2(\bar{x}^h(\omega)) \right\|_{p\text{-}var;[0,T]}\le C\left\|  S_2(\bar{x}^h)(\omega) \right\|_{\frac1p\text{-}{\rm H\ddot{o}l};[0,T]}
	 \le  C(\omega,p,T),\quad a.s.
	\end{align*}
\end{proof}

For the $\tilde{N}$-truncated modified equation
\begin{equation}\label{TME1}
\left\{
\begin{aligned}
\dot{\tilde{y}}^{\tilde{N}}_{t}&=\sum_{|\alpha|=1}^{\tilde{N}}f_\alpha(\tilde{y}_{t}^{\tilde{N}})h^{\alpha_0-1}\Delta^{\alpha_1}_{1,1}\cdots\Delta^{\alpha_d}_{1,d},\quad t\in[0,h];\\
\tilde{y}_{0}^{\tilde{N}}&=z,
\end{aligned}
\right.
\end{equation}
we prove that there exists some truncated number $\tilde{N}=\tilde{N}(h)$ such that the local error is exponentially small with respect to the time step size $h$, which answers Problem \ref{P3}.

\begin{theorem}\label{tm-local}
	Let $0<\epsilon<\frac12$. Assume that $V_l$, $l=0,1,\cdots,d$ are analytic on the closed ball $$B_{2R}(z):=\left\{y\in \mathbb{C}^m: \|y-z\|\le 2R \right\}$$ with 
	\begin{align*}
	\|V_l(y)\|\le M, \quad \forall ~y\in B_{2R}(z).
	\end{align*}
	 Then for sufficiently small $h$, there exists $\tilde{N}=\tilde{N}(h)$ such that 
	\begin{align*}
	\|\tilde{y}^{\tilde{N}}_{t_1}-Y^h_1\|\le Che^{- h_0/h^{\frac12-\epsilon} },
	\end{align*}
	where $\tilde{y}^{\tilde{N}}_{t_1}$ is the solution of \eqref{TME1} and $Y^h_1$ is defined by the one-step numerical method \eqref{TrunRK}
\end{theorem}
 
 Before proving Theorem \ref{tm-local}, we recall the Cauchy's estimate for analytic functions, and give four lemmas about estimates for the truncated increments, the numerical solution, the modified equation and the truncated modified equation, respectively.

 \begin{lemma}(Cauchy's estimate)
 	Suppose that $f$ is analytic on a neighbourhood of the closed ball $B(y^*,R)$ and 
$M_R=\max\{|f(y)|:y\in B(y^*,R) \}<\infty$,
 	then 
 	\begin{align*}
 	f^{(n)}(y^*)\le \frac{n!M_R}{R^n}.
 	\end{align*}
 \end{lemma}

 \begin{lemma}(estimate for $\Delta_{n+1,l}$)\label{lm-h1}
	Let $0<\epsilon<\frac12$ and $k\ge 1$. Then 
	there exists a constant $C=C(\epsilon,k)$ such that 
	\begin{align}\label{h1}
	|\Delta_{n+1,l}|\le h^{\frac12-\epsilon}, \quad \forall~h<C.
	\end{align}
\end{lemma}
\begin{proof}
	Consider the function $v_1(h)=k\ln h +h^{-2\epsilon}$. Then $v'_1(h)=\frac{k}{h}-\frac{2\epsilon}{h^{2\epsilon+1}}$ satisfies $v'_1<0$ as $h\rightarrow0$, and $v'_1>0$ as $h\rightarrow\infty$. Moreover, $v'_1(h)=0$ with $h=\left(\frac{2\epsilon}{k}\right)^{\frac{1}{2\epsilon}}$. Combining the fact that $v_1(h)\ge 0$ as $h\rightarrow 0$, we obtain that there exists a constant $C=C(\epsilon,k)$ such that 
	\begin{align*}
	|\zeta_{n+1,l}|\le h^{-\epsilon}, \quad \forall~h<C,
	\end{align*}
	which implies \eqref{h1}.
\end{proof}

  \begin{lemma}(estimate for $d_\alpha$) \label{lm-dj}
  	Denote $\kappa := \max _{i=1,\cdots,s}\left\{\sum_{j=1}^{s}|a_{ij}|\right\}$ and $\mu := \sum_{i=1}^{s}|b_i|$. Under assumptions as in Theorem \ref{tm-local}, if 
  	\begin{align}\label{ap-hDelta}
  	\max\{h,|\Delta_{1,1}|,\cdots,|\Delta_{1,d}|\}<\frac{R}{2\kappa M (d+1)\sqrt{s}},
  	\end{align}
  	then it holds that
 	\begin{align*}
 	\|d_\alpha(y)\|\le\mu (d+1) M \left[\frac{2\kappa M (d+1)\sqrt{s}}{R}\right]^{|\alpha|-1},\quad \forall~y\in B_R(z),
 	\end{align*}
 	where the coefficient $d_\alpha$ is defined by the expansion
 	\begin{align*}
 	Y^h_1(z)=z+\sum_{|\alpha|=1}^{\infty}d_\alpha(z)h^{\alpha_0}\Delta_{1,1}^{\alpha_1}\cdots\Delta_{1,d}^{\alpha_d}, \quad \alpha=(\alpha_0,\cdots,\alpha_d)\in\mathbb{N}^{d+1}.
 	\end{align*}
 \end{lemma}
 
 \begin{proof}
For any $y\in B_{\frac32 R}(z)$ and $\|\Delta y\|\le 1$, define $v(\theta):=V_l(y+\theta\Delta y)$, $|\theta|\le \frac{R}{2}$. Then Cauchy's estimate shows
\begin{align*}
\|V'_l(y)\Delta y\|=\left\|v'(\theta)\big|_{\theta=0}\right\|\le \frac{M}{\frac{R}{2}}=\frac{2M}{R},
\end{align*}
which implies
\begin{align}\label{V'}
\|V'_l(y)\|=\sup_{\|\Delta y\|\le 1}\|V'_l(y)\Delta y\|\le \frac{2M}{R}, \quad \forall~y\in B_{\frac32 R}(z).
\end{align}

For any $y\in B_{ R}(z)$, define a map $F:\mathbb{C}^{m\times s}\rightarrow \mathbb{C}^{m\times s}$ by
\begin{align*}
F:g&=(g_1,\cdots,g_s)\mapsto F(g)=(F(g)_1,\cdots,F(g)_s),\\
F(g)_i &=y  +\sum^{s}_{j=1}a_{ij}\left[V_0(g_j)h+\sum_{l=1}^{d}V_l(g_j)\Delta_{1,l}\right],\quad i=1,\cdots,s.
\end{align*}
We claim that $F$ is a contraction on the closed set $\mathbb{B}:=\left\{(g_1,\cdots,g_s):\|g_i-y\|\le \frac{R}{2}, i=1,\cdots,s\right\}$.
Indeed, for any $0<\gamma<1$ and 
\begin{align*}
\max\{h,|\Delta_{1,1}|,\cdots,|\Delta_{1,d}|\}\le\frac{\gamma R}{2\kappa M (d+1)\sqrt{s}}=:C_1(\gamma),
\end{align*}
we have
\begin{align*}
\|F(g)_i-y\|\le\sum^{s}_{j=1}|a_{ij}|M\left[h+\sum_{l=1}^{d}|\Delta_{1,l}|\right]<\frac{R}{2}, \quad \forall~g\in \mathbb{B}.
\end{align*}
Besides, \eqref{V'} yields
\begin{align*}
\|F(g)_i-F(\tilde{g})_i\|\le \sum^{s}_{j=1}|a_{ij}|\frac{2M}{R}\left[h+\sum_{l=1}^{d}|\Delta_{1,l}|\right]\|g-\tilde{g}\|
\le\frac{\gamma}{\sqrt{s}}\|g-\tilde{g}\|, \quad \forall~g,\tilde{g}\in \mathbb{B},
\end{align*}
which leads to $\|F(g)-F(\tilde{g})\|\le \gamma\|g-\tilde{g}\|$. Therefore, there exists a unique fixed point $g^*$ for $F$ on the set $\mathbb{B}$.
Since $g^*_i\in B_{\frac{3R}{2}}(z)$ and $Y^h_1(y)=y+\sum^{s}_{i=1}b_{i}\left(V_0(g^*_i)h+\sum_{l=1}^{d}V_l(g^*_i)\Delta_{1,l}\right)$,
the boundedness of $V$ deduces
\begin{align*}
\|Y^h_1(y)-y\|\le \mu (d+1)MC_1(\gamma).
\end{align*}
Repeatly applying Cauchy's estimate, we have
\begin{align*}
\|d_\alpha(y)\|&=\left\|\frac{1}{\alpha_0!\cdots \alpha_d!}\left[\frac{d^{\alpha_d}}{d\Delta_{1,d}^{\alpha_d}}\cdots\left[\frac{d^{\alpha_0}
	}{dh^{\alpha_0}}\left( \Phi_h(y) -y\right)\right]\Big\rvert _{h=0}\cdots\right]\bigg\rvert _{\Delta_{1,d}=0}\right\|\\
&\le \frac{\mu (d+1) M C_1(\gamma)}{C_1^{|\alpha|}(\gamma)}=\mu (d+1) M \left[\frac{2\kappa M (d+1)\sqrt{s}}{\gamma R}\right]^{|\alpha|-1}.
\end{align*}
Letting $\gamma\rightarrow 1$, we obtain 
\begin{align*}
\|d_\alpha(y)\|\le\mu (d+1) M \left[\frac{2\kappa M (d+1)\sqrt{s}}{R}\right]^{|\alpha|-1}, \quad  \forall ~ y\in B_{ R}(z).
\end{align*}
\end{proof}

\begin{remark}
Let $\epsilon=\frac14$ and $k=4$. Lemma \ref{lm-h1} shows that condition \eqref{ap-hDelta} holds if we simulate the random variable $\Delta_{1,l}$ in \eqref{Delta} by taking
\begin{align}\label{h2}
h< \min\left\{C(\epsilon,k),\left[ \frac{ R}{2\kappa M (d+1)\sqrt{s}} \right]^4\right\}.
\end{align}
\end{remark}

\begin{lemma}(estimate for $f_\alpha$)\label{lm-fj}
Denote $\eta:=2\max\{\kappa,\mu/(2\ln 2 -1)\}$. Under assumptions as in Theorem \ref{tm-local} and Lemmas \ref{lm-h1}-\ref{lm-dj}, then the coefficients of the associated stochastic modified equation \eqref{Delmod} satisfy
\begin{align*}
\sum_{|\alpha|=J}\|f_\alpha(y)\|\le (\ln 2) \eta M (d+1)^2 \sqrt{s}\left(\frac{\eta M(d+1)^2 \sqrt{s}J}{R}\right)^{J-1},\quad \forall ~y\in B_{ \frac12 R}(z),~ J\in\mathbb{N}_+.
\end{align*}
\end{lemma}
\begin{proof}
	For $J=1$, it follows from Lemma \ref{lm-dj} that 
	\begin{align*}
	\sum_{|\alpha|=1}\|f_\alpha(y)\|\le \mu M(d+1)^2  \le (\ln 2) \eta M (d+1)^2 \sqrt{s}, \quad \forall ~y\in B_{ \frac12 R}(z).
	\end{align*}
	
	For $J\ge 2$, we consider $\alpha\in\mathbb{N}^{d+1}$ such that $1\le |\alpha|\le J$.
Define 
\begin{align*}
\|f\|_{|\alpha|}:= \max\left\{\|f(y)\|:y\in B_{R-(|\alpha|-1)\delta}(z),~\delta=\frac{R}{2(J-1)}\right\}.
\end{align*}
For $|k^{i,1}|+\cdots+|k^{i,i}|=|\alpha|$, $i=1,\cdots,|\alpha|$, we use \cite[Lemma 7.4 in Chap.  \uppercase\expandafter{\romannumeral9}]{GeometricEH} to get
\begin{align*}
\|D_{k^{i,1}}\cdots D_{k^{i,i-1}}f_{k^{i,i}}\|_{|\alpha|}&\le \frac{1}{\delta}\|f_{k^{i,1}}\|_{|\alpha|}\|D_{k^{i,2}}\cdots D_{k^{i,i-1}}f_{k^{i,i}}\|_{|\alpha|-1}\\
&\le \frac{1}{\delta^2}\|f_{k^{i,1}}\|_{|\alpha|}\|f_{k^{i,1}}\|_{|\alpha|-1}\|D_{k^{i,3}}\cdots D_{k^{i,i-1}}f_{k^{i,i}}\|_{|\alpha|-2}\\
&\le \cdots\\
&\le \frac{1}{\delta^{i-1}}\|f_{k^{i,1}}\|_{|\alpha|}\|f_{k^{i,1}}\|_{|\alpha|-1}\cdots\| f_{k^{i,i}}\|_{|\alpha|-(i-1)}\\
&\le \frac{1}{\delta^{i-1}}\|f_{k^{i,1}}\|_{|k^{i,1}|}\cdots\|f_{k^{i,i}}\|_{|k^{i,i}|}.
\end{align*}
Combining with \eqref{confj}, we have
\begin{align*}
\|f_\alpha\|_{|\alpha|}\le \|d_\alpha\|_{|\alpha|}+\sum_{i=2}^{|\alpha|}\frac{1}{i!}\sum_{(k^{i,1},\cdots,k^{i,i})\in O_i}\frac{1}{\delta^{i-1}}\|f_{k^{i,1}}\|_{|k^{i,1}|}\cdots\|f_{k^{i,i}}\|_{|k^{i,i}|}.
\end{align*}
We introduce the notation  $F_{\tilde{\alpha}}:=\sum_{|\alpha|=\tilde{\alpha}}\|f_\alpha\|_{|\alpha|}$ and $G_{\tilde{\alpha}}:=\sum_{|\alpha|=\tilde{\alpha}}\|d_\alpha\|_{|\alpha|}$, thus the above inequality yields
\begin{align}\label{betaj}
F_{\tilde{\alpha}}\le G_{\tilde{\alpha}} + \sum_{i=2}^{\tilde{\alpha}}\frac{1}{i!}\sum_{\tilde{k}^{i,1}+\cdots+\tilde{k}^{i,i}=\tilde{\alpha}}\frac{1}{\delta^{i-1}}
F_{\tilde{k}^{i,1}}\cdots F_{\tilde{k}^{i,i}}.
\end{align}

Notice that Lemma \ref{lm-dj} implies 
\begin{align*}
G_{\tilde{\alpha}}&\le \frac{(d+\tilde{\alpha})!}{d!\tilde{\alpha}!} \mu (d+1) M \left[\frac{2\kappa M (d+1)\sqrt{s}}{R}\right]^{\tilde{\alpha}-1}\\
&\le (d+1)^{\tilde{\alpha}} \mu (d+1) M \sqrt{s}\left[\frac{2\kappa M (d+1)\sqrt{s}}{R}\right]^{\tilde{\alpha}-1}\\
&=\mu M(d+1)^2\sqrt{s}\left[\frac{2\kappa M(d+1)^2\sqrt{s}}{R}\right]^{\tilde{\alpha}-1}.
\end{align*}
We let 
\begin{align*}
\beta _{\tilde{\alpha}}:=\frac{\mu M(d+1)^2\sqrt{s}}{\delta}\left(\frac{2\kappa M(d+1)^2\sqrt{s}}{R}\right)^{\tilde{\alpha}-1}+\sum_{i=2}^{\tilde{\alpha}}\frac{1}{i!}\sum_{\tilde{k}^{i,1}+\cdots+\tilde{k}^{i,i}=\tilde{\alpha}}
\beta_{\tilde{k}^{i,1}}\cdots \beta_{\tilde{k}^{i,i}}, 
\end{align*}
which is defined for all $\tilde{\alpha}\in\mathbb{N}_+$.
According to \eqref{betaj}, we know that $F_{\tilde{\alpha}}\le \delta \beta _{\tilde{\alpha}}$ if $\tilde{\alpha}\le J$. In order to estimate $F_{J}$,  it suffices to estimate $\beta _{J}$.
Let $c_1:=\frac{\mu M(d+1)^2\sqrt{s}}{\delta}$, $c_2:=\frac{2\kappa M(d+1)^2\sqrt{s}}{R}$ and 
\begin{align}\label{b}
b(\xi):=\sum_{\tilde{\alpha}=1}^{\infty}\beta_{\tilde{\alpha}}\xi^{\tilde{\alpha}}.
\end{align}
To apply Cauchy's estimate for analyzing $\beta_{J}$, we give a bound on $b(\xi)$.
Notice that the definition of 
$\beta _{\tilde{\alpha}}$ leads to 
\begin{align*}
b(\xi)=\frac{c_1 \xi}{ 1-c_2\xi}+e^{b(\xi)}-1-b(\xi).
\end{align*}
Consider the function 
\begin{align*}
q(b,\xi)=\frac{c_1 \xi}{ 1-c_2\xi}+e^{b}-1-2b=0.
\end{align*}
If $\frac{\partial q(b,\xi)}{\partial b}=e^b-2\neq 0$ (i.e., $b\neq \ln 2+2k\pi i$), the implicit function theorem is applicable, and indeed, the relationship between $b$ and $\xi$ is exactly given \eqref{b}. Since $c_1,c_2>0$, we know that for any $\xi\in\mathbb{R}$ such that $\xi\in[0,(2\ln2-1)/(c_1+c_2(2\ln2-1)))$, the function $\frac{c_1 \xi}{ 1-c_2\xi}\in[0,2\ln 2-1)$ is increasing with respect to $\xi$. Meanwhile, for any $b\in\mathbb{R}$ such that $b\in[0,\ln 2)$, the function $-e^{b}+1+2b\in[0,2\ln 2-1)$ is increasing with respect to $b$. Then we have $b(\xi)\in[0,\ln 2]$ if $\xi\in[0,(2\ln2-1)/(c_1+c_2(2\ln2-1)))$. Combining with $\beta_{\tilde{\alpha}}>0$, we obtain 
\begin{align*}
|b|\le \sum_{\tilde{\alpha}=1}^{\infty}\beta_{\tilde{\alpha}}|\xi|^{\tilde{\alpha}}\le \ln 2, \quad \forall~ |\xi|< (2\ln2-1)/(c_1+c_2(2\ln2-1)).
\end{align*}
Since Cauchy's estimate is applicable, we derive
\begin{align*}
|\beta_{\tilde{\alpha}}|\le \frac{\ln 2}{((2\ln2-1)/(c_1+c_2(2\ln2-1)))^{\tilde{\alpha}}}, \quad \tilde{\alpha}\in\mathbb{N}_+,
\end{align*}
and then
\begin{align*}
F_J\le \delta \beta_J \le\frac{R}{2(J-1)}\frac{\ln 2}{((2\ln2-1)/(c_1+c_2(2\ln2-1)))^{J}}\le
 \frac{(\ln 2)R}{2(J-1)}\left(\frac{\eta M(d+1)^2\sqrt{s} J}{R}\right)^J.
\end{align*}
Therefore, 
\begin{align*}
\sum_{|\alpha|=J}\|f_\alpha(y)\|\le (\ln 2) \eta M (d+1)^2 \sqrt{s}\left(\frac{\eta M(d+1)^2 \sqrt{s}J}{R}\right)^{J-1},\quad \forall ~y\in B_{ \frac12 R}(z).
\end{align*}
\end{proof}

In order to estimate the exact solution of the $\tilde{N}$-truncated modified equation \eqref{TME1}, 
we consider the infinite expansion for its solution with respect to the initial value $z$:
\begin{align*}
\tilde{y}^{\tilde{N}}_{t_1}=z+\sum_{|\alpha|=1}^{\infty}f^{\tilde{N}}_\alpha(z)h^{\alpha_0}\Delta^{\alpha_1}_{1,1}\cdots\Delta^{\alpha_d}_{1,d}.
\end{align*}

\begin{lemma}(estimate for $f^{\tilde{N}}_\alpha$)\label{lm-fjN}
	Let $0<\epsilon<\frac12$. Under assumptions as in Theorem \ref{tm-local} and Lemmas \ref{lm-h1}-\ref{lm-fj}, if the truncation number $\tilde{N}$ satisfies
	\begin{align}\label{N}
	1\le\tilde{N}\le \frac{R}{\eta M(d+1)^2\sqrt{s}h^{\frac12-\epsilon}},
	\end{align}
	then there exists a constant $C$ such that 
	\begin{align*}
	\|f^{\tilde{N}}_\alpha(z)\|\le  \frac{ (\ln 2) \eta M (d+1)^2\sqrt{s}C}{\left[\frac{R}{2(\ln 2)\eta M (d+1)^2\sqrt{s}C}\right]^{(\frac{1}{1/2-\epsilon})|\alpha|-1}}.
	\end{align*}
\end{lemma}
\begin{proof}
For simplicity, we let $\epsilon=\frac14$, as the proof is similar for $\epsilon\in(0,\frac12)$.
	
According to Lemma \ref{lm-fj}, as long as $\{\tilde{y}^{\tilde{N}}_{t}:t\le t_1=h\}\subset B_{\frac{R}{2}}(z)$, we have the estimate
\begin{align*}
\|\tilde{y}^{\tilde{N}}_{t}-z\|
&\le  \sum_{J=1}^{\tilde{N}}h^{\frac14 J}
(\ln 2) \eta M (d+1)^2 \sqrt{s}\left(\frac{\eta M(d+1)^2 \sqrt{s}J}{R}\right)^{J-1}\\
&\le h^{\frac14} (\ln 2) \eta M (d+1)^2 \sqrt{s} \left(1+   \sum_{J=2}^{\tilde{N}}\left(\frac{\eta M(d+1)^2 \sqrt{s}Jh^{\frac14}}{R}\right)^{J-1} \right), \quad \forall ~t\le h.
\end{align*}
Since $1\le\tilde{N}\le \frac{R}{\eta M(d+1)^2\sqrt{s}h^{\frac14}}$,
we know 
 \begin{align*}
1+   \sum_{J=2}^{\tilde{N}}\left(\frac{\eta M(d+1)^2 \sqrt{s}Jh^{\frac14}}{R}\right)^{J-1}\le 1+   \sum_{J=2}^{\tilde{N}}\left(\frac{J}{\tilde{N}}\right)^{J-1}\le C_0.
 \end{align*}
Then a sufficient condition for $\{\tilde{y}^{\tilde{N}}_{t}:t\le h\}\subset B_{\frac{R}{2}}(z)$ is
\begin{align}\label{h3}
h\le \left(\frac{R}{2(\ln 2) \eta M (d+1)^2\sqrt{s} C_0}\right)^4.
\end{align}
In this case, it has
\begin{align*}
\|\tilde{y}^{\tilde{N}}_{t}-z\|
\le h^{\frac14} (\ln 2) \eta M (d+1)^2 \sqrt{s} C_0, \quad \forall ~t\le h.
\end{align*}

Combining the conditions \eqref{h2} and \eqref{h3} on $h$ together, we obtain that there exists a sufficently large $C$ such that
\begin{align*}
\left[\frac{R}{2(\ln 2)\eta M (d+1)^2\sqrt{s}C}\right]^4\le\min \left\{C(\epsilon,k),\left[\frac{\gamma R}{2\kappa M (d+1)\sqrt{s}}\right]^4,\left[\frac{R}{2(\ln 2)\eta M (d+1)^2\sqrt{s}C_0}\right]^4 \right\}.
\end{align*}
Defining $C_2:=\left[\frac{R}{2(\ln 2)\eta M (d+1)^2\sqrt{s}C}\right]^4$, we use Cauchy's  estimate to get
\begin{align*}
\|f^{\tilde{N}}_\alpha(z)\|&=\frac{1}{\alpha_0!\cdots \alpha_d!}\left[\frac{d^{\alpha_d}}{d\Delta_{1,d}^{\alpha_d}}\cdots\left[\frac{d^{\alpha_0}
}{dh^{\alpha_0}}\left( \tilde{y}^{\tilde{N}}_h -z\right)\right]\Big\rvert _{h=0}\cdots\right]\bigg\rvert _{\Delta_{1,d}=0}\\
&\le  \frac{C_2^{\frac14} (\ln 2) \eta M (d+1)^2\sqrt{s}C_0}{C_2^{|\alpha|}}\\
&\le  \frac{ (\ln 2) \eta M (d+1)^2\sqrt{s}C}{\left[\frac{R}{2(\ln 2)\eta M (d+1)^2\sqrt{s}C}\right]^{4|\alpha|-1}}.
\end{align*}
\end{proof}

Now we can proceed to the proof of Theorem \ref{tm-local}.

\begin{proof}[Proof of Theorem \ref{tm-local}]
We know that $d_\alpha=f_\alpha^{\tilde{N}}$ with $1\le|\alpha|\le\tilde{N}$, then it remains to estimate the terms for $|\alpha|\ge \tilde{N}+1$. For simplicity, we let $\epsilon=\frac14$, since the proof is similar for $0<\epsilon<\frac12$.

For the numerical solution given by \eqref{TrunRK}, Lemma \ref{lm-dj} yields that the sum of remainder terms is bounded by
\begin{align*}
&\sum_{|\alpha|=\tilde{N}+1}^{\infty}\|d_\alpha(z)\|h^{\alpha_0}|\Delta^{\alpha_1}_{1,1}|\cdots|\Delta^{\alpha_d}_{1,d}|\\
&\le \sum_{|\alpha|=\tilde{N}+1}^{\infty} (d+1)^J\mu (d+1) M \left[\frac{2\kappa M (d+1)\sqrt{s}}{R}\right]^{J-1}h^{\frac{J}{4}}\\
&\le \left\{\sum_{J=0}^{\infty} h^{\frac{J}{4}} \left[\frac{2\kappa M (d+1)\sqrt{s}}{R}\right]^{J-1}    \right\} \mu (d+1)^2 M  \left[\frac{2\kappa M (d+1)\sqrt{s}}{R}\right]^{\tilde{N}+1} h^{\frac{\tilde{N}+1}{4}}  \\
&\le C\tilde{C}^{\tilde{N}}h^{\frac{\tilde{N}+1}{4}}.
\end{align*}
The last inequality holds if $h^{\frac{1}{4}} \left[\frac{2\kappa M (d+1)\sqrt{s}}{R}\right]   \le\gamma<1$, i.e., $h< \left[\frac{\gamma R}{2\kappa M (d+1)\sqrt{s}}\right]^{4}   $.

For the exact solution of the $\tilde{N}$-truncated modified equation \eqref{TME1}, Lemma \ref{lm-fjN} leads to that the sum of remainder terms is bounded by
\begin{align*}
&\sum_{|\alpha|=\tilde{N}+1}^{\infty}\|f^{\tilde{N}}_\alpha(z)\|h^{\alpha_0}|\Delta^{\alpha_1}_{1,1}|\cdots|\Delta^{\alpha_d}_{1,d}|\\
&\le \sum_{J=\tilde{N}+1}^{\infty} (d+1)^J   \frac{ (\ln 2) \eta M (d+1)^2C_0\sqrt{s}}{\left[\frac{R}{2(\ln 2)\eta M (d+1)^2C_0\sqrt{s}}\right]^{4J-1}}h^{\frac{J}{4}} \\
&\le \left\{\sum_{J=0}^{\infty} \frac{h^{\frac{J}{4}}(d+1)^J }{\left[\frac{R}{2(\ln 2)\eta M (d+1)^2C_0\sqrt{s}}\right]^{4J-1}}  \right\} 
 \frac{ (\ln 2) \eta M (d+1)^2C_0\sqrt{s}   (d+1)^{\tilde{N}+1} }{\left[\frac{R}{2(\ln 2)\eta M (d+1)^2C_0\sqrt{s}}\right]^{4(\tilde{N}+1)}}h^{\frac{\tilde{N}+1}{4}}\\
 &\le C\tilde{C}^{\tilde{N}}h^{\frac{\tilde{N}+1}{4}}.
\end{align*}
The last inequality holds if $\frac{(d+1) h^{\frac{1}{4}}}{\left[\frac{R}{2(\ln 2)\eta M (d+1)^2C_0\sqrt{s}}\right]^{4}}\le\gamma<1$, i.e., $h<\left(\frac{\gamma}{d+1}\right)^4 \left[\frac{\gamma R}{2(\ln 2)\eta M (d+1)^2C_0\sqrt{s}}\right]^{16}$.

Since condition \eqref{N} reads $\tilde{N}\le \frac{R}{\eta M(d+1)^2\sqrt{s}h^{\frac14}}:=h_0h^{-\frac14}$, we choose $\tilde{N}$ for the largest integer under this condition and then
\begin{align*}
C\tilde{C}^{\tilde{N}}h^{\frac{\tilde{N}+1}{4}}=C\tilde{C}^3h\tilde{C}^{\tilde{N}-3}h^{\frac{1}{4}(\tilde{N}-3)}\le C\tilde{C}^3h\left(\tilde{C}h^{\frac{1}{4}}\right)^{\tilde{N}-3}.
\end{align*}
Due to $h_0h^{-\frac14}<\tilde{N}+1$, we have 
\begin{align*}
\left(\tilde{C}h^{\frac{1}{4}}\right)^{\tilde{N}-3}\le e^{-(N-3)}\le e^4e^{-(N+1)}\le e^4e^{- h_0/h^{\frac14} },\quad \forall~h\le (\tilde{C}e)^{-4}.
\end{align*}
Therefore, when $h$ is sufficiently small, the local error is 
\begin{align*}
\|\tilde{y}^{\tilde{N}}_{t_1}-Y^h_1\|\le Che^{- h_0/h^{\frac14} }.
\end{align*}
\end{proof}

\section{Numerical experiments}\label{sec5}

Numerical experiments are carried out based on three rough Hamiltonian systems in this section. Based on Examples \ref{ex1}-\ref{ex2}, we verify the convergence orders proved in  Theorem \ref{tm-Nglobal} and Corollary \ref{tm-Nglobaladd} for multiplicative and additive cases, accordingly. 
In Example \ref{ex3}, which is a linear system with the energy conservation law, we present the long time behavior of several numerical methods and the corresponding modified equations. 

\begin{example}\label{ex1}
\begin{equation*}
\left\{
\begin{aligned}
dP_t&=\sin(P_t)\sin(Q_t)dt-\cos(Q_t) dX_t^{2},\quad P_{0}=p,\\
dQ_t&=\cos(P_t)\cos(Q_t)dt-\sin(P_t) dX_t^{1},\quad Q_{0}=q,
\end{aligned}
\right.
\end{equation*}
where $X^{1}$ and $X^{2}$ are independent fBms with Hurst parameter $H\in(1/4,1/2]$. The Hamiltonians are 
\begin{align*}
\mathcal{H}_{0}(P_t,Q_t)=\sin(P_t)\cos(Q_t),\quad \mathcal{H}_{1}(P_t,Q_t)=\cos(P_t),\quad \mathcal{H}_{2}(P_t,Q_t)=\sin(Q_t).
\end{align*}
\end{example}

\begin{example}(flow driven by the Taylor--Green velocity field \cite[Corollary 4.3]{WXZ})\label{ex2} 
\begin{equation*}
\left\{
\begin{aligned}
	dP_t&=-\sin(Q_t)dt+\sqrt{2}\sigma dX_t^{1},\quad P_{0}=p,\\
	dQ_t&=\sin(P_t)dt+\sqrt{2}\sigma dX_t^{2},\quad\quad Q_{0}=q,
\end{aligned}
\right.
\end{equation*}
	where $X^{1}$ and $X^{2}$ are independent fBms with Hurst parameter $H\in(1/4,1/2]$. The Hamiltonians are 
	\begin{align*}
	\mathcal{H}_{0}(P_t,Q_t)=-\cos(P_t)-\cos(Q_t),\quad \mathcal{H}_{1}(P_t,Q_t)=-\sqrt{2}\sigma Q_t,\quad \mathcal{H}_{2}(P_t,Q_t)=\sqrt{2}\sigma P_t.
	\end{align*}
\end{example}

We consider the midpoint scheme 
\begin{align}\label{M1}
Y^h_{n+1}=Y^h_n+V\left(\frac{Y^h_n+Y^h_{n+1}}{2}\right)X_{t_n,t_{n+1}},
\end{align}
whose $2$-truncated and $4$-truncated modified equations are defined via the following formulas for the coefficients
\begin{align*}
|\alpha|&=1:\quad f_\alpha(y)=V_\alpha(y);\\
|\alpha|&=2:\quad f_\alpha(y)=0;\\
|\alpha|&=3:\quad f_\alpha(y)=\sum_{\alpha_1+\alpha_2+\alpha_3=\alpha}\left[-\frac{1}{24}V''_{\alpha_3}(y)V_{\alpha_2}(y)V_{\alpha_1}(y)+\frac{1}{12}V'_{\alpha_3}(y)V'_{\alpha_2}(y)V_{\alpha_1}(y)\right];\\
|\alpha|&=4:\quad f_\alpha(y)=0.
\end{align*}

We apply the midpoint scheme to Example \ref{ex1} with the initial datum $(p,q)=(1,0)$ and the time interval $[0,T]=[0,1]$. Figure \ref{fi-mul} plots the mean-square errors $\|Y^h_N-\tilde{y}^{2}_{T}\|_{L^2(\Omega)}$ and $\|Y^h_N-\tilde{y}^{4}_{T}\|_{L^2(\Omega)}$, where the time step sizes are $h=2^{-i}$, $i=4,5,6,7,8$ and the Hurst parameters are $H=0.4,0.45,0.5$. For each time step size $h$, the `exact' solution of a truncated modified equation is simulated by using the midpoint scheme to this modified equation with a tiny step size $\delta=2^{-12}$. The expectation is approximated by $200$ sample trajectories. The convergence orders of  are revealed to be $3H-1$ and $5H-1$, respectively, from which we confirm the result in Theorem \ref{tm-Nglobal} for the multiplicative case. 
In Example \ref{ex2}, we take $p=1$, $q=0$, $\sigma=2$ and $T=1$, and choose $H=0.3,0.4,0.5$. Then the convergence orders of $\|Y^h_N-\tilde{y}^{2}_{T}\|_{L^2(\Omega)}$ and $\|Y^h_N-\tilde{y}^{4}_{T}\|_{L^2(\Omega)}$ are $2H$ and $4H$, respectively, which verifies the results in Corollary \ref{tm-Nglobaladd} for the additive case.
Furthermore, one can find out that the numerical solution is closer to the exact solution of the $4$-truncated modified equation than that of the $2$-truncated modified equation.

\begin{figure}[H]
	\centering
	\subfigure{
		\begin{minipage}[t]{0.3\linewidth}
			\includegraphics[height=4.3cm,width=4.3cm]{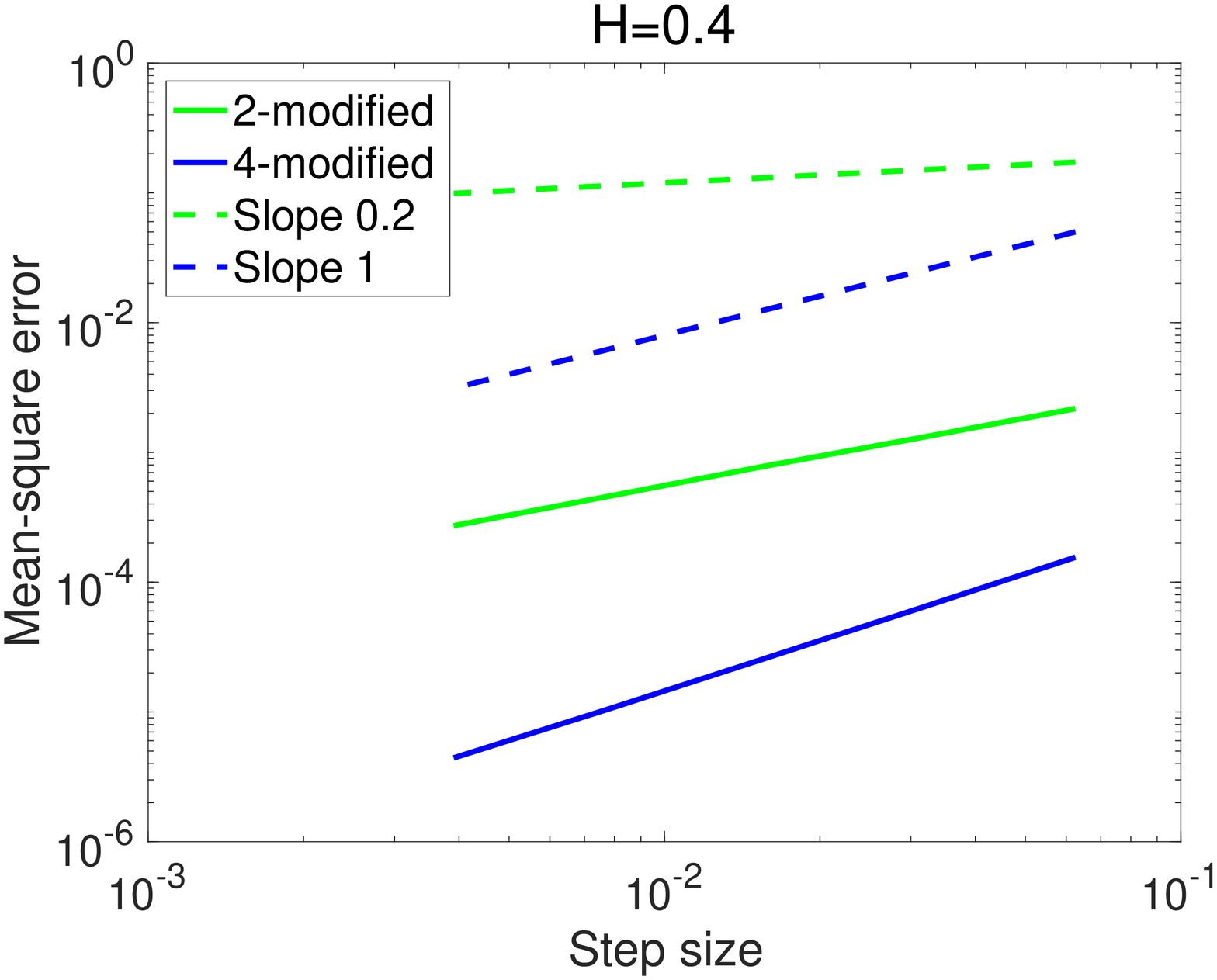}
		\end{minipage}
	}
	\subfigure{
		\begin{minipage}[t]{0.3\linewidth}
			\includegraphics[height=4.3cm,width=4.3cm]{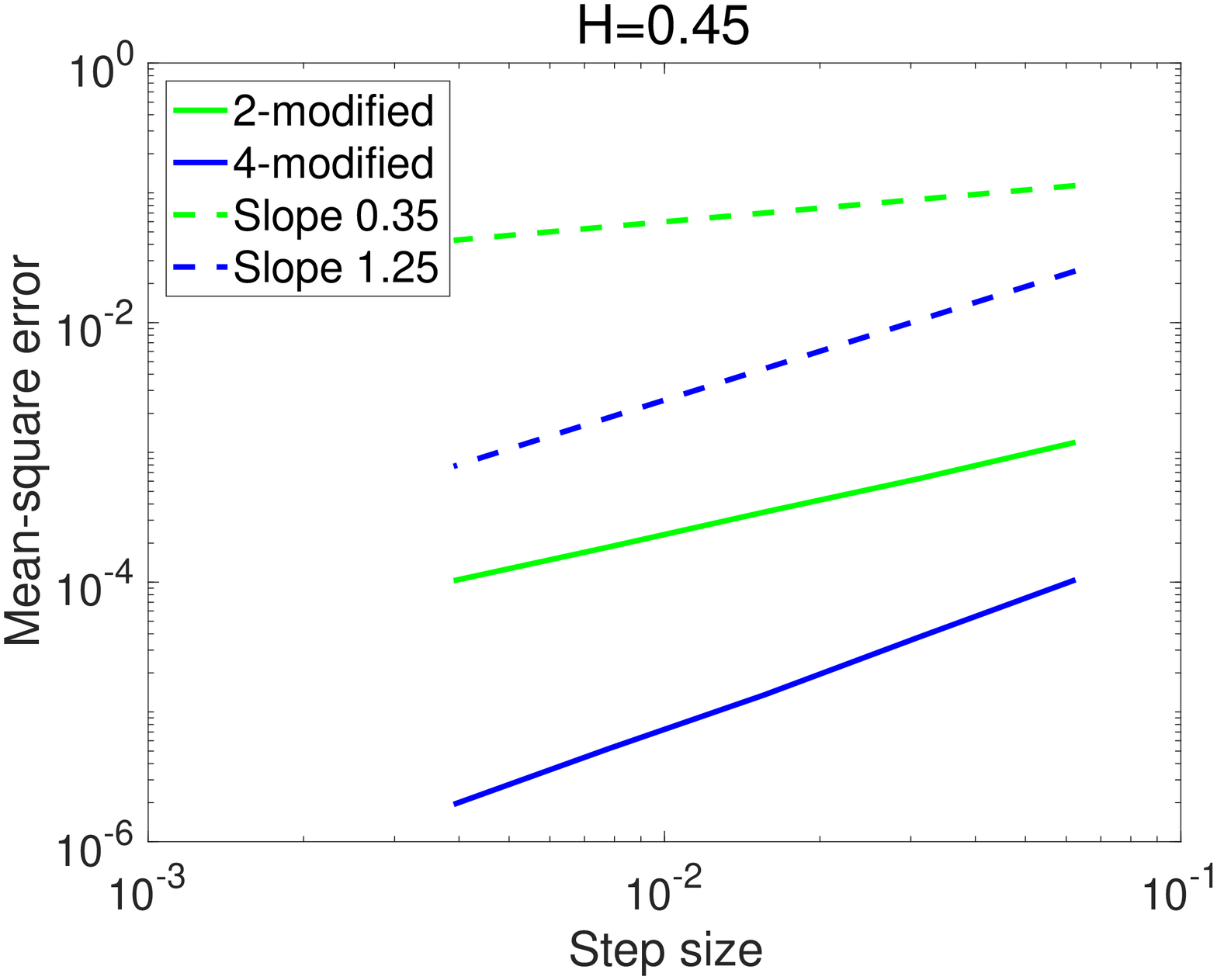}
		\end{minipage}
	}
	\subfigure{
		\begin{minipage}[t]{0.3\linewidth}
			\includegraphics[height=4.3cm,width=4.3cm]{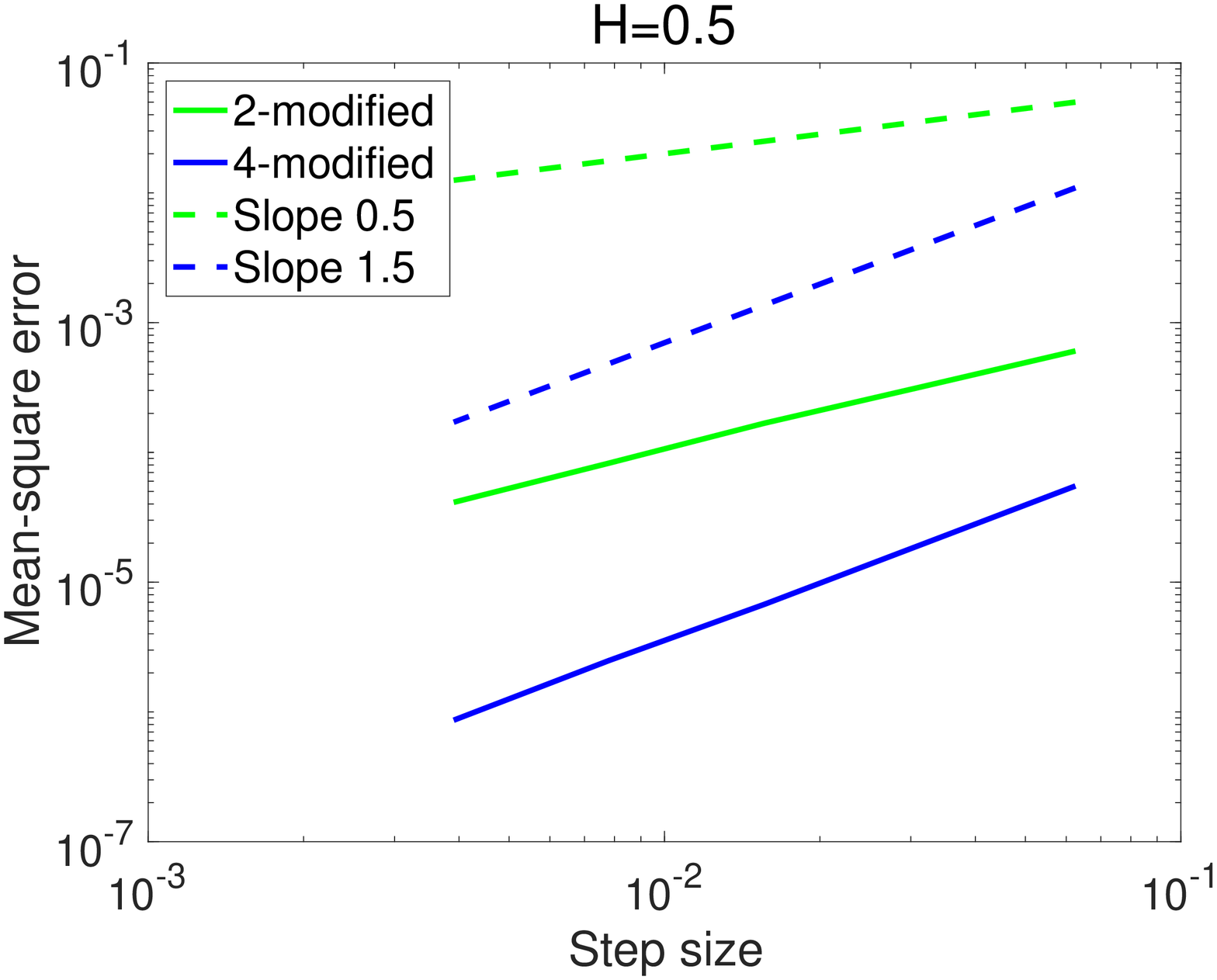}
		\end{minipage}
	}
	\caption{Mean-square error vs. Step size for Example \ref{ex1}}\label{fi-mul}
\end{figure}

\begin{figure}[H]
	\centering
	\subfigure{
		\begin{minipage}[t]{0.3\linewidth}
			\includegraphics[height=4.3cm,width=4.3cm]{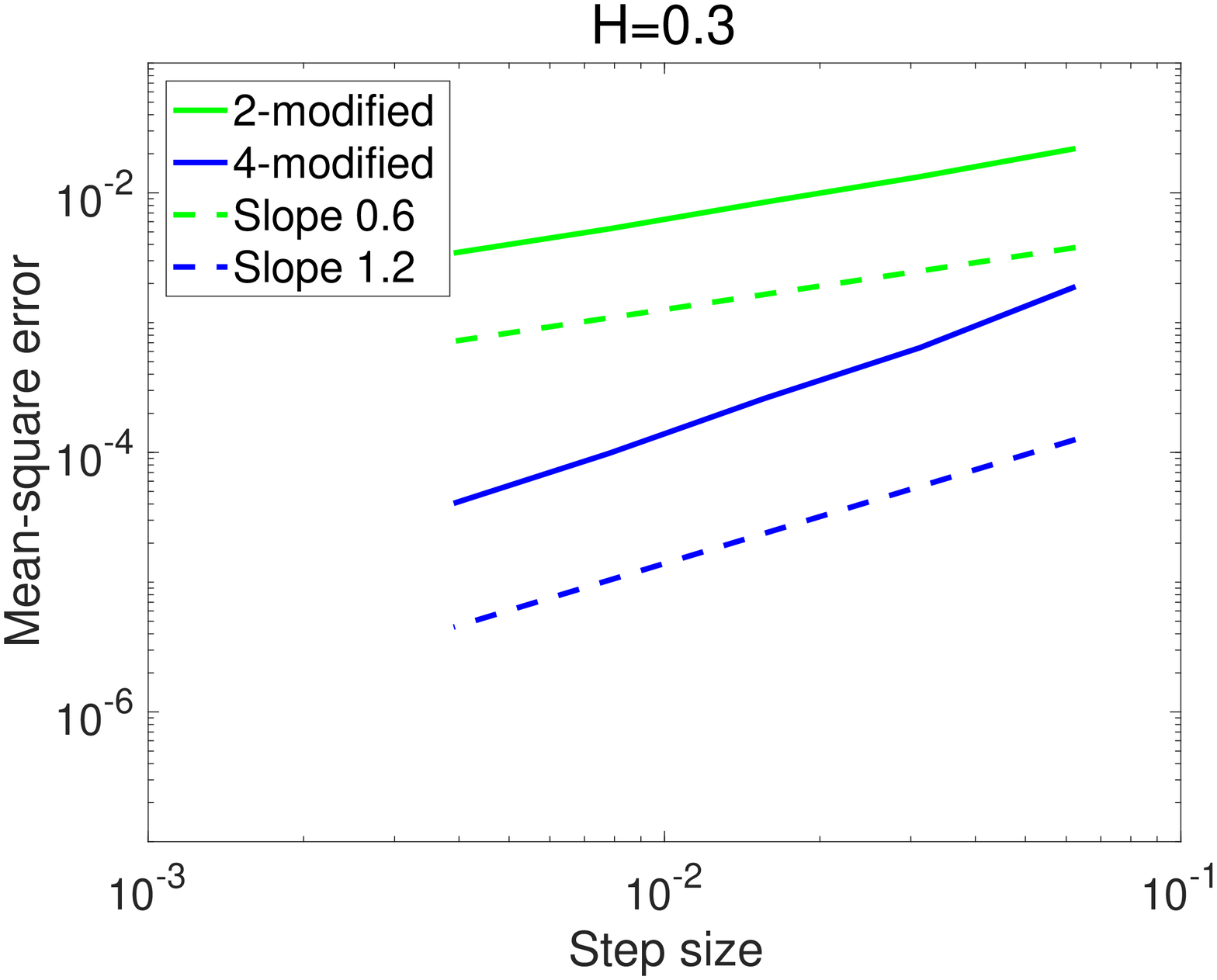}
		\end{minipage}
	}
	\subfigure{
		\begin{minipage}[t]{0.3\linewidth}
			\includegraphics[height=4.3cm,width=4.3cm]{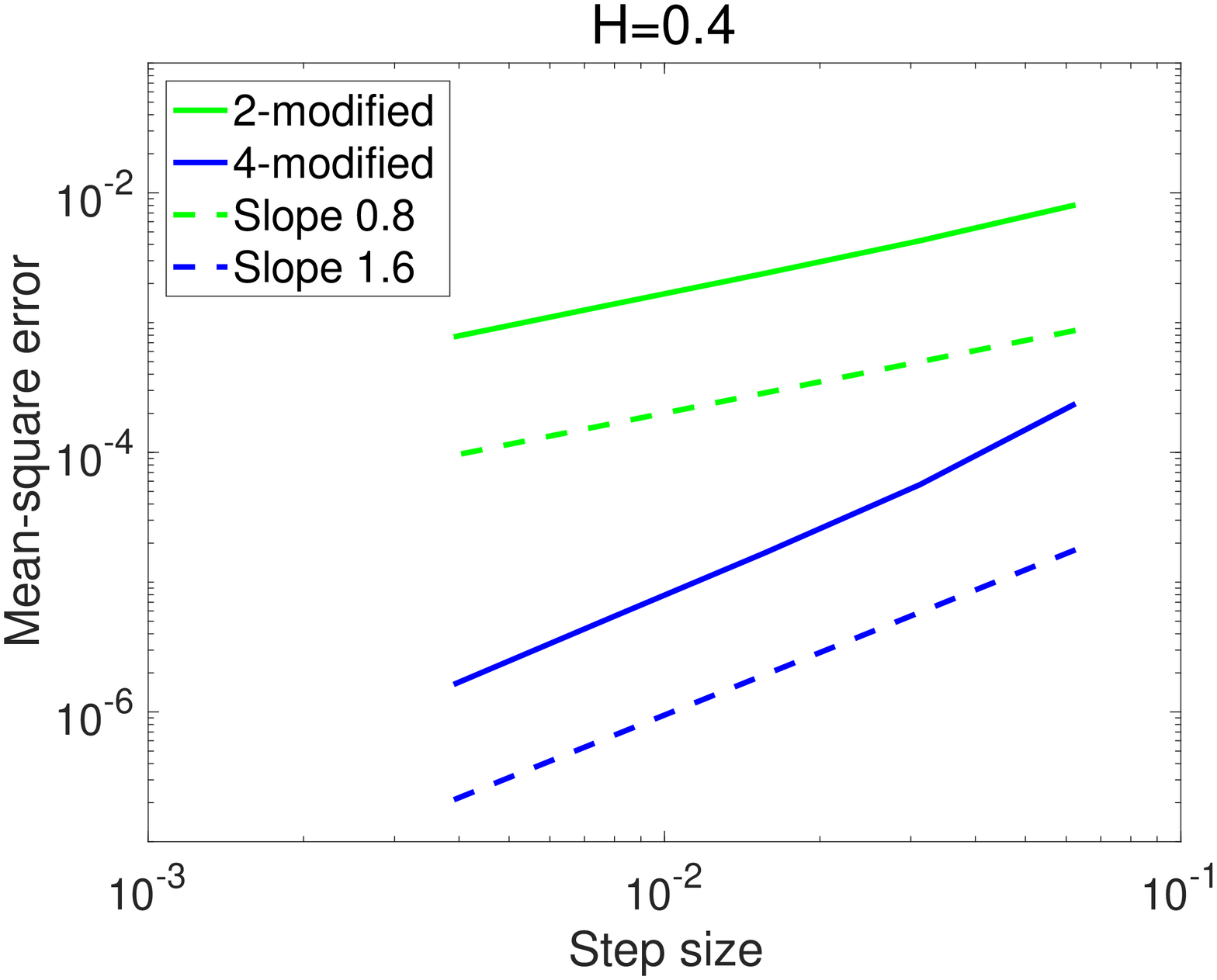}
		\end{minipage}
	}
	\subfigure{
		\begin{minipage}[t]{0.3\linewidth}
			\includegraphics[height=4.3cm,width=4.3cm]{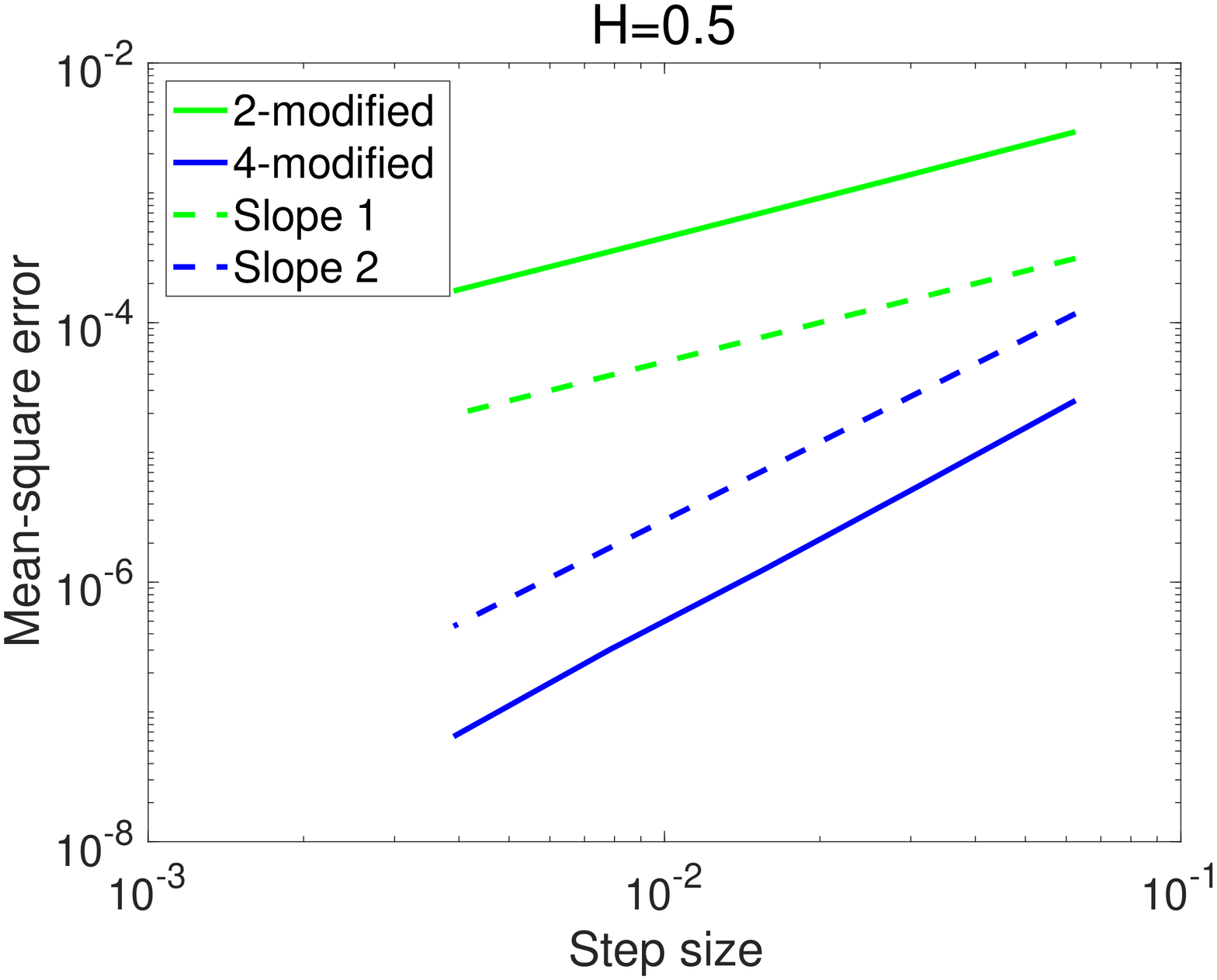}
		\end{minipage}
	}
	\caption{Mean-square error vs. Step size for Example \ref{ex2}}\label{fi-add}
\end{figure}

\begin{example}(Kubo Oscillator \cite{HHW18})\label{ex3}
	\begin{equation*}
	\left\{
	\begin{aligned}
	dP_t=&-aQ_tdt-\sigma\sum^{2}_{i=1}Q_td X_t^{i},\quad P_{0}=p,\\
	dQ_t=&aP_tdt+\sigma\sum^{2}_{i=1}P_td X_t^{i},\quad \quad ~Q_{0}=q,
	\end{aligned}
\right.
\end{equation*}
	where $X^{1}$ and $X^{2}$ are independent standard Brownian motions.
	The Hamiltonians are
	\begin{align*}
	\frac{2}{a}\mathcal{H}_{0}(P_t,Q_t)=\frac{2}{\sigma} \mathcal{H}_{1}(P_t,Q_t)=\frac{2}{\sigma}\mathcal{H}_{2}(P_t,Q_t)=P_t^2+Q_t^2.
	\end{align*}
	Note that $\mathcal{H}(P_t,Q_t)=P_t^2+Q_t^2$ is an invariant.
	The exact solution reads
	\begin{equation*}
\left\{
\begin{aligned}
	P_t=p\cos(a t+\sigma\sum^{2}_{i=1}X_{t}^{i})-q\sin(a t+\sigma\sum^{2}_{i=1}X_{t}^{i}),\\
	Q_t=q\cos(a t+\sigma\sum^{2}_{i=1}X_{t}^{i})+p\sin(a t+\sigma\sum^{2}_{i=1}X_{t}^{i}).
	\end{aligned}
\right.
\end{equation*}
\end{example}

We compare the midpoint scheme \eqref{M1}, which is symplectic and energy-preserving, with the following two numerical methods. One is an explicit RK method defined by 
\begin{align}\label{M2}
Y^h_{n+1}=Y^h_n+V\left(Y^h_n+\frac12 V(Y^h_n)X_{t_n,t_{n+1}}\right)X_{t_n,t_{n+1}},
\end{align}
which is neither symplectic and nor energy-preserving.
The associated $2$-truncated and $4$-truncated modified equations are defined through the formulas for the coefficients:
\begin{align*}
|\alpha|&=1:\quad f_\alpha(y)=V_\alpha(y);\\
|\alpha|&=2:\quad f_\alpha(y)=0;\\
|\alpha|&=3:\quad  f_\alpha(y)=\sum_{\alpha_1+\alpha_2+\alpha_3=\alpha}\left[-\frac{1}{24}V''_{\alpha_3}(y)V_{\alpha_2}(y)V_{\alpha_1}(y)-\frac{1}{6}V'_{\alpha_3}(y)V'_{\alpha_2}(y)V_{\alpha_1}(y)\right];\\
|\alpha|&=4:\quad f_\alpha(y)=\sum_{\alpha_1+\alpha_2+\alpha_3+\alpha_4=\alpha}\left[\frac{1}{12}V'_{\alpha_4}V''_{\alpha_3}(y)V_{\alpha_2}(y)V_{\alpha_1}(y)+\frac{1}{8}V'_{\alpha_4}(y)V'_{\alpha_3}(y)V'_{\alpha_2}(y)V_{\alpha_1}(y)\right].
\end{align*}
Another one is a symplectic partitioned RK method which is not energy-preserving. Applying it to Example \ref{ex3} leads to (see also \cite[Section 5.1]{Milstein})
\begin{equation}\label{M3}
\left\{
\begin{aligned}
P^h_{n+1}&=P^h_n-aQ^h_nh-\sigma^2P^h_{n+1}h-\sigma\sum_{i=1}^{2} Q^h_nX^i_{t_n,t_{n+1}},\\
Q^h_{n+1}&=Q^h_n+aP^h_{n+1}h+\sigma^2Q^h_{n}h+\sigma\sum_{i=1}^{2} P^h_{n+1}X^i_{t_n,t_{n+1}}.
\end{aligned}
\right.
\end{equation}
The coefficients of the associated modified equations for $1\le|\alpha|\le3 $ are calculated as follows. Denote $y=(y^1,y^2)^\top\in\mathbb{R}^2$, then
$$|\alpha|=1:\quad f_{(1,0,0)}(y)=\left( \begin{array}{cc}
-\sigma^2 & -a \\
a & \sigma^2\end{array} \right)\left( \begin{array}{c}
y^1 \\
y^2\end{array}\right), \qquad
f_{(0,1,0)}(y)=f_{(0,0,1)}(y)=\left( \begin{array}{cc}
0 & -\sigma \\
\sigma & 0\end{array} \right)\left( \begin{array}{c}
y^1 \\
y^2\end{array} \right);$$
$$|\alpha|=2:\quad f_{(2,0,0)}(y)=\left( \begin{array}{cc}
\frac{\sigma^4}{2}+\frac{a^2}{2} & a\sigma^2 \\
-a\sigma^2 & -\frac{\sigma^4}{2}-\frac{a^2}{2}\end{array} \right)\left( \begin{array}{c}
y^1 \\
y^2\end{array}\right), \qquad
f_{(0,1,1)}(y)=\left( \begin{array}{cc}
\sigma^2 &0 \\
0 & -\sigma^2\end{array} \right)\left( \begin{array}{c}
y^1 \\
y^2\end{array} \right),$$
$$f_{(1,1,0)}(y)=f_{(1,0,1)}(y)=\left( \begin{array}{cc}
a\sigma & \sigma^3 \\
-\sigma^3 & -a\sigma\end{array} \right)\left( \begin{array}{c}
y^1 \\
y^2\end{array} \right),~~
f_{(0,2,0)}(y)=f_{(0,0,2)}(y)=\left( \begin{array}{cc}
\frac{\sigma^2}{2} & 0 \\
0 & -\frac{\sigma^2}{2}\end{array} \right)\left( \begin{array}{c}
y^1 \\
y^2\end{array} \right);$$
$$|\alpha|=3:\quad f_{(3,0,0)}(y)=\left( \begin{array}{cc}
-\frac{\sigma^6}{3}-\frac{2a^2\sigma^2}{3} & -\frac{5a\sigma^4}{6}-\frac{a^3}{6} \\
\frac{5a\sigma^4}{6}+\frac{a^3}{6}& \frac{\sigma^6}{3}+\frac{2a^2\sigma^2}{3} \end{array} \right)\left( \begin{array}{c}
y^1 \\
y^2\end{array}\right), ~
f_{(1,1,1)}(y)=\left( \begin{array}{cc}
-\frac{4\sigma^4}{3} & -a\sigma^2 \\
a\sigma^2 & \frac{4\sigma^4}{3} \end{array} \right)
\left( \begin{array}{c}
y^1 \\
y^2\end{array} \right),$$
$$f_{(1,2,0)}(y)=f_{(1,0,2)}(y)=\left( \begin{array}{cc}
-\sigma^4 & -\frac{2a\sigma^2}{3} \\
\frac{2a\sigma^2}{3} & \sigma^4\end{array} \right)\left( \begin{array}{c}
y^1 \\
y^2\end{array}\right),\qquad
f_{(0,2,1)}(y)=f_{(0,1,2)}(y)=\left( \begin{array}{cc}
0 & -\frac{\sigma^3}{2} \\
\frac{\sigma^3}{2}  & 0\end{array} \right)\left( \begin{array}{c}
y^1 \\
y^2\end{array} \right),$$
$$f_{(2,1,0)}(y)=f_{(2,0,1)}(y)=\left( \begin{array}{cc}
-\frac{4a\sigma^3}{3} & -\frac{5\sigma^5}{6}-\frac{a^2\sigma}{2}  \\
\frac{5\sigma^5}{6}+\frac{a^2\sigma}{2} & \frac{4a\sigma^3}{3} \end{array} \right)\left( \begin{array}{c}
y^1 \\
y^2\end{array} \right),~
f_{(0,3,0)}(y)=f_{(0,0,3)}(y)=\left( \begin{array}{cc}
0 & -\frac{\sigma^3}{6} \\
\frac{\sigma^3}{6}  & 0\end{array} \right)\left( \begin{array}{c}
y^1 \\
y^2\end{array} \right).$$

We set $a=1$, $\sigma=0.9$, $T=20$, $N=10\times2^6$ (i.e., $h=\frac{T}{N}=0.0313$). We present the evolution of domains under the flow of $Y^h_n(z)$, $Y_{t_n}(z)$ and $\tilde{y}^{\tilde{N}}_{t_n}(z)$ with $n=0,75,100,180$, for one realization of Example \ref{ex3} in Figure \ref{Kubosym}. For methods \eqref{M1}-\eqref{M2}, the truncation numbers are $\tilde{N}=2,4$. For method \eqref{M3}, $\tilde{N}=2,3$.
The `exact' solution of a truncated modified equation is simulated by applying the midpoint scheme to this modified equation with a tiny step size $\delta=\frac{T}{10\times2^{15}}=2^{-14}$.
Notice the fact that the preservation of the symplectic structure is equivalent to the preservation of the area of domains in $2$-dimensional case. The areas of domains remain unchanged under symplectic methods \eqref{M1} and \eqref{M3}, as well as those given by the flows of associated truncated modified equations. However, the corresponding areas for method \eqref{M2} and its $4$-truncated modified equation increase. In particular, we point out that the $2$-truncated modified equation of methods \eqref{M2} possesses the symplectic conservation law, since it coincides with the Wong--Zakai approximation of the original system and shares the same formula as the $2$-truncated modified equation of method \eqref{M1}. These numerical results support Theorem \ref{tmfj}-\ref{tmsym}.

In Figures \ref{err_M1}-\ref{err_M3}, we perform simulations for a trajectory with $a=1$, $\sigma=1$, $p=1$, $q=0$, $T=50$, $N=10\times2^8$ (i.e., $h=\frac{T}{N}=0.0195$) by the three methods, successively. The errors $\|Y^h_n-Y_{t_n}\|$ and $\|Y^h_n-\tilde{y}^{\tilde{N}}_{t_n}\|$ are given in Figures \ref{a1}-\ref{a3}. The `exact' solution of a truncated modified equation is simulated by applying the midpoint scheme to this modified equation with a tiny step size $\delta=\frac{T}{10\times2^{15}}$. As expected, we see that the error decreases as $\tilde{N}$ becomes larger for a numerical method. Besides, the energy errors  $|(Y^h_n)^\top Y^h_n -p^2-q^2|$ and $|(\tilde{y}^{\tilde{N}}_{t_n})^\top \tilde{y}^{\tilde{N}}_{t_n} -p^2-q^2|$ are presented in Figures \ref{b1}-\ref{b3}. Noting that the energy-preserving method \eqref{M1} is also a symmetry method, we have that $f_\alpha(y)=0$ for any $|\alpha|=2k$, $k\in\mathbb{N}_+$. Therefore, what we observe is that the energy error is almost zero for method \eqref{M1} and its truncated modified equations. As to the other two methods, the energy is not preserved, but the energy error is generally controlled better by the symplectic method \eqref{M3} than by non-symplectic method \eqref{M2}.

\begin{figure}
	\centering
	\subfigure[The midpoint scheme \eqref{M1}]{
		\begin{minipage}[t]{0.47\linewidth}
			\includegraphics[height=8cm,width=8cm]{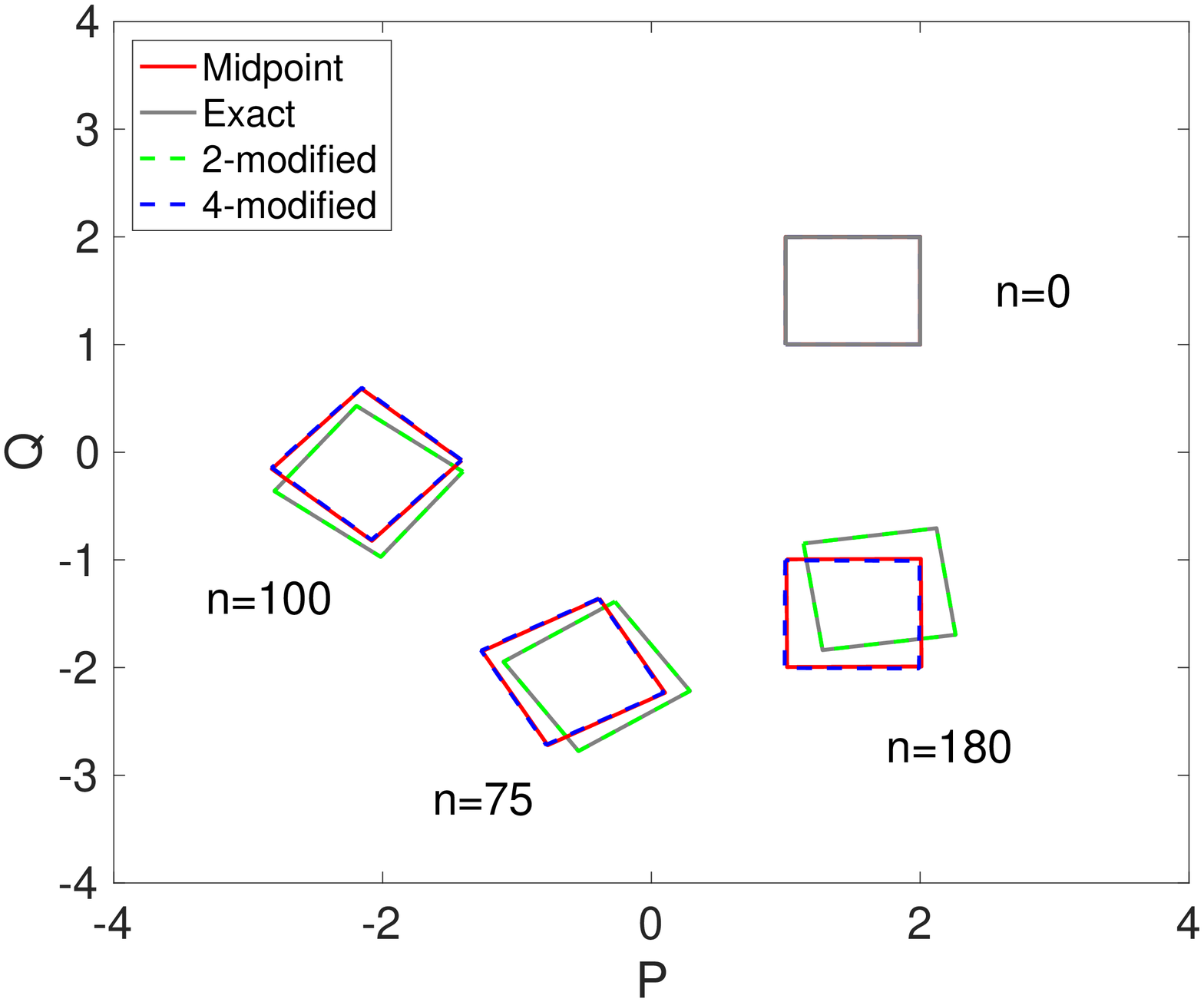}
		\end{minipage}
	}
	\subfigure[The explicit RK method \eqref{M2}]{
		\begin{minipage}[t]{0.47\linewidth}
			\includegraphics[height=8cm,width=8cm]{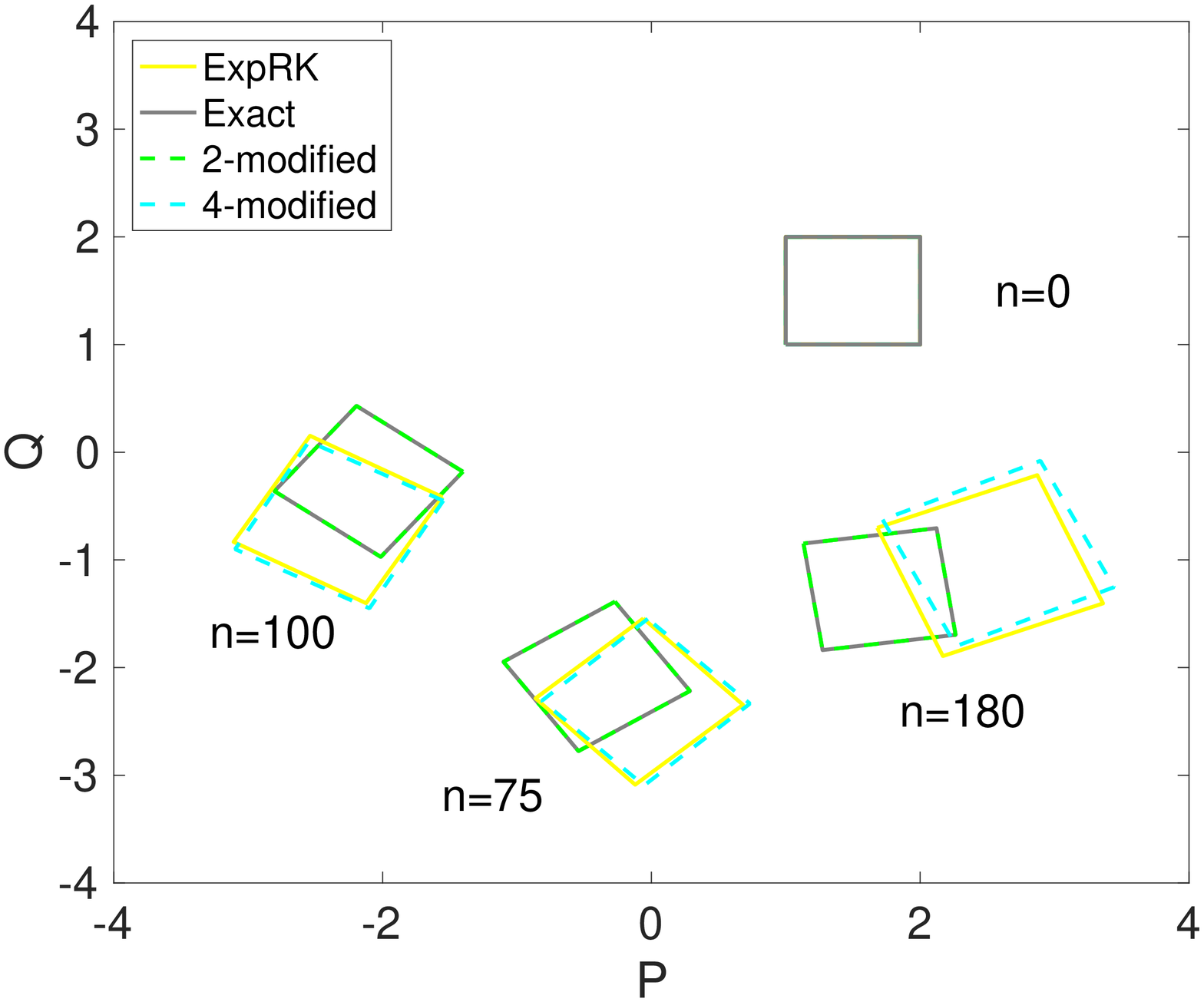}
		\end{minipage}
	}
	\subfigure[The symplectic partitioned RK method \eqref{M3}]{
	\begin{minipage}[t]{0.47\linewidth}
		\includegraphics[height=8cm,width=8cm]{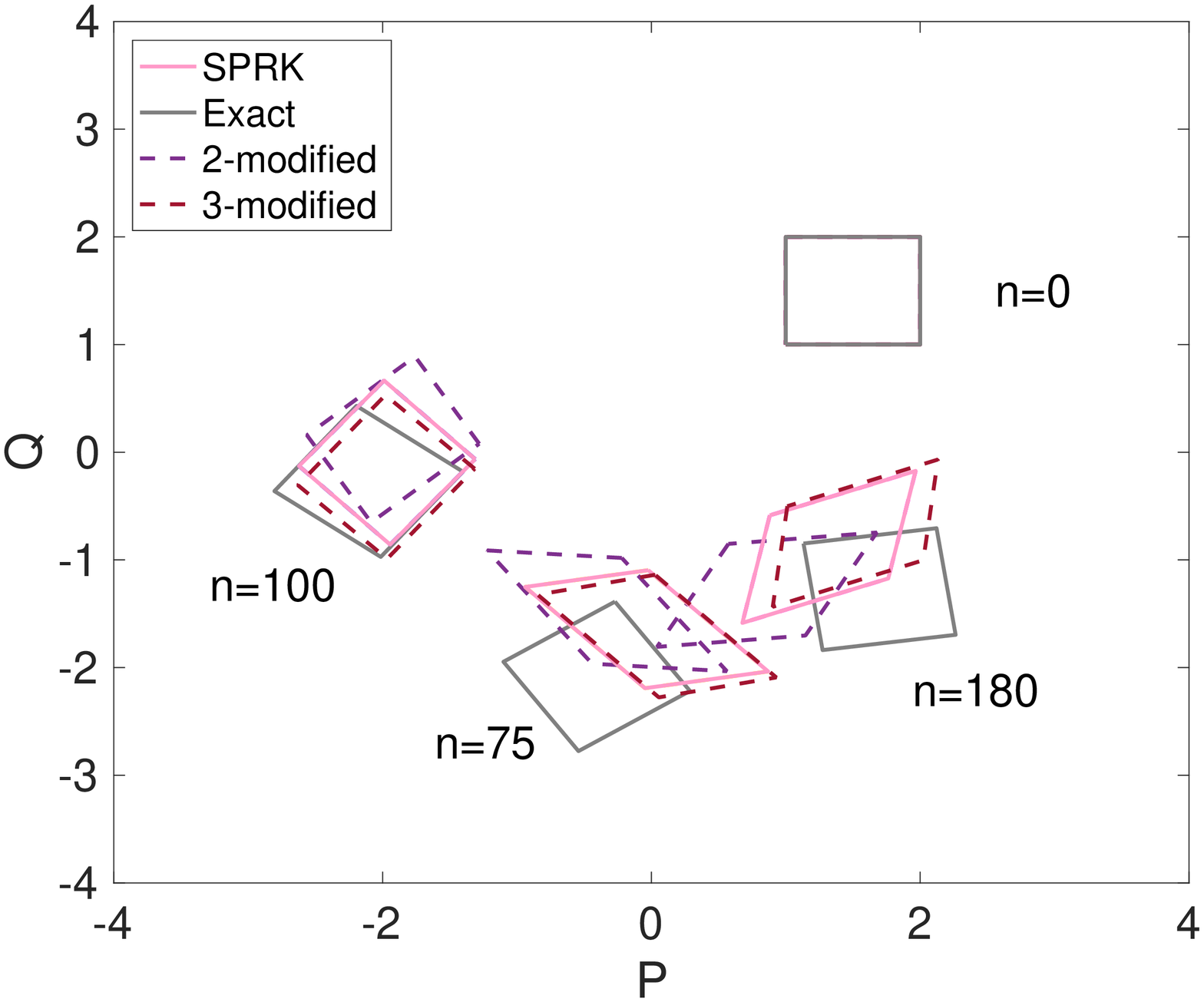}
	\end{minipage}
}
	\caption{Evolution of domains in the phase plane}\label{Kubosym}
\end{figure}

\begin{figure}
	\centering
	\subfigure[Error for one trajectory]{
		\begin{minipage}[t]{0.47\linewidth}
			\includegraphics[height=6cm,width=8cm]{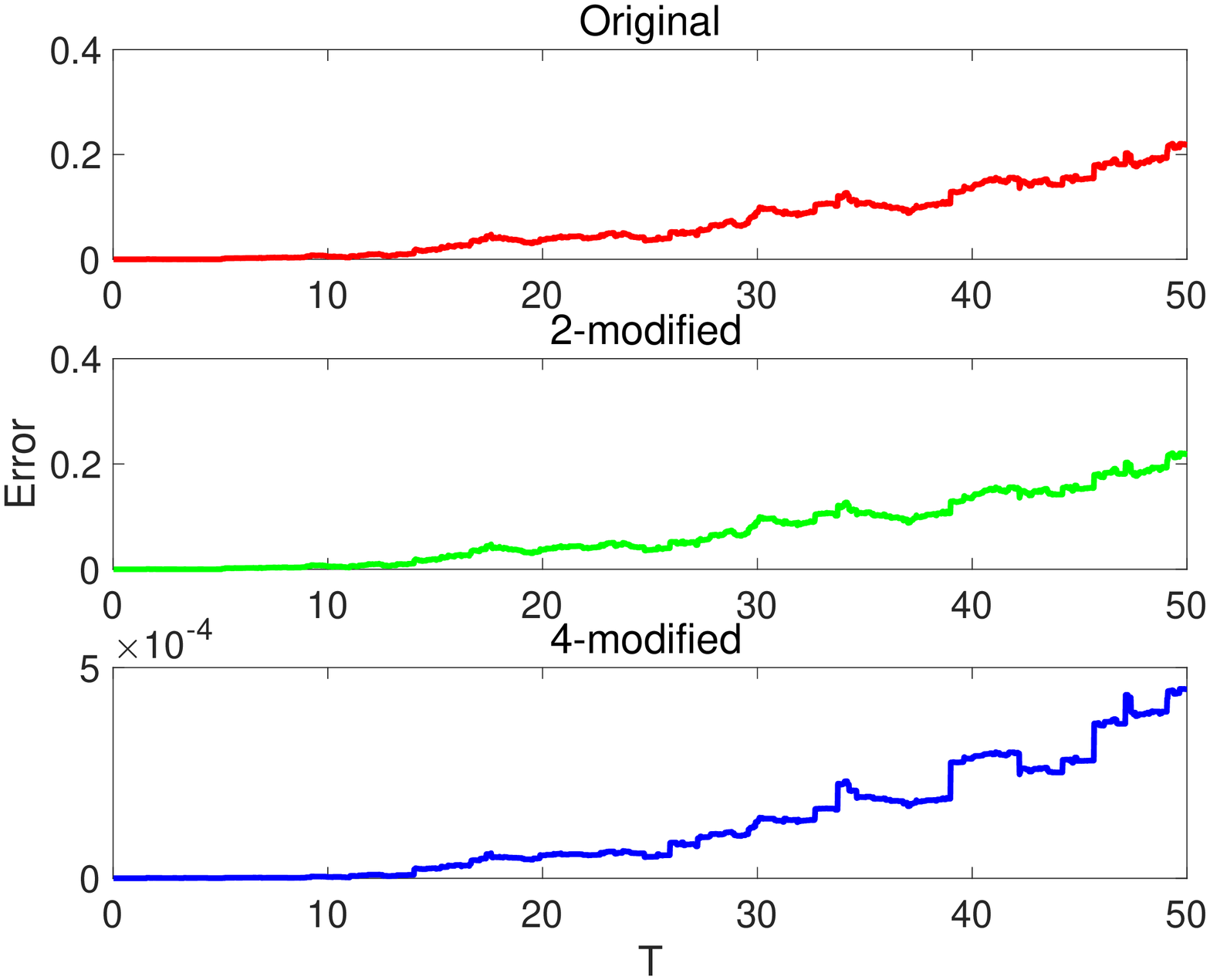}\label{a1}
		\end{minipage}
	}
	\subfigure[Energy error for one trajectory]{
		\begin{minipage}[t]{0.47\linewidth}
			\includegraphics[height=6cm,width=8cm]{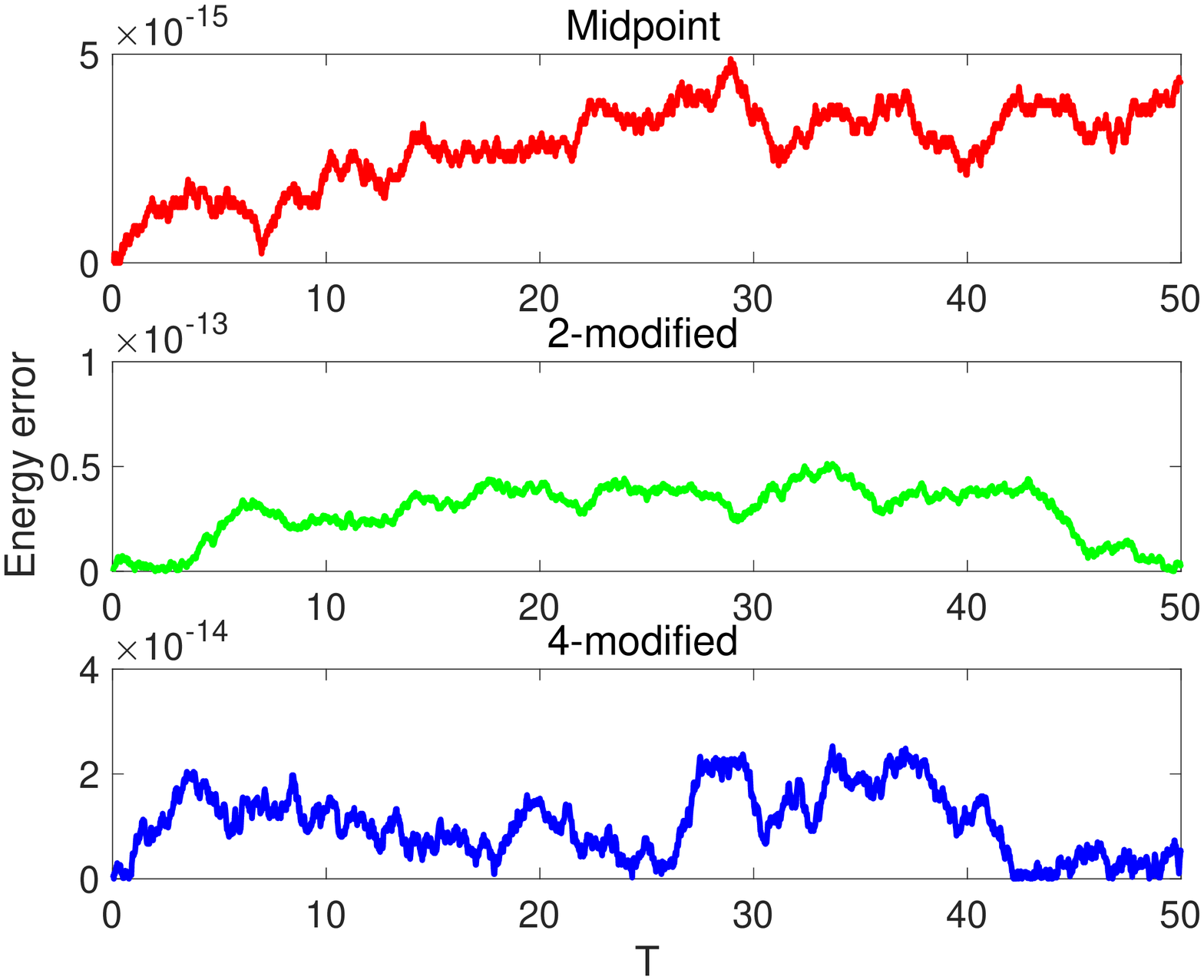}\label{b1}
		\end{minipage}
	}
	\caption{The midpoint scheme \eqref{M1}}\label{err_M1}
\end{figure}

\begin{figure}
	\centering
	\subfigure[Error for one trajectory]{
		\begin{minipage}[t]{0.47\linewidth}
			\includegraphics[height=6cm,width=8cm]{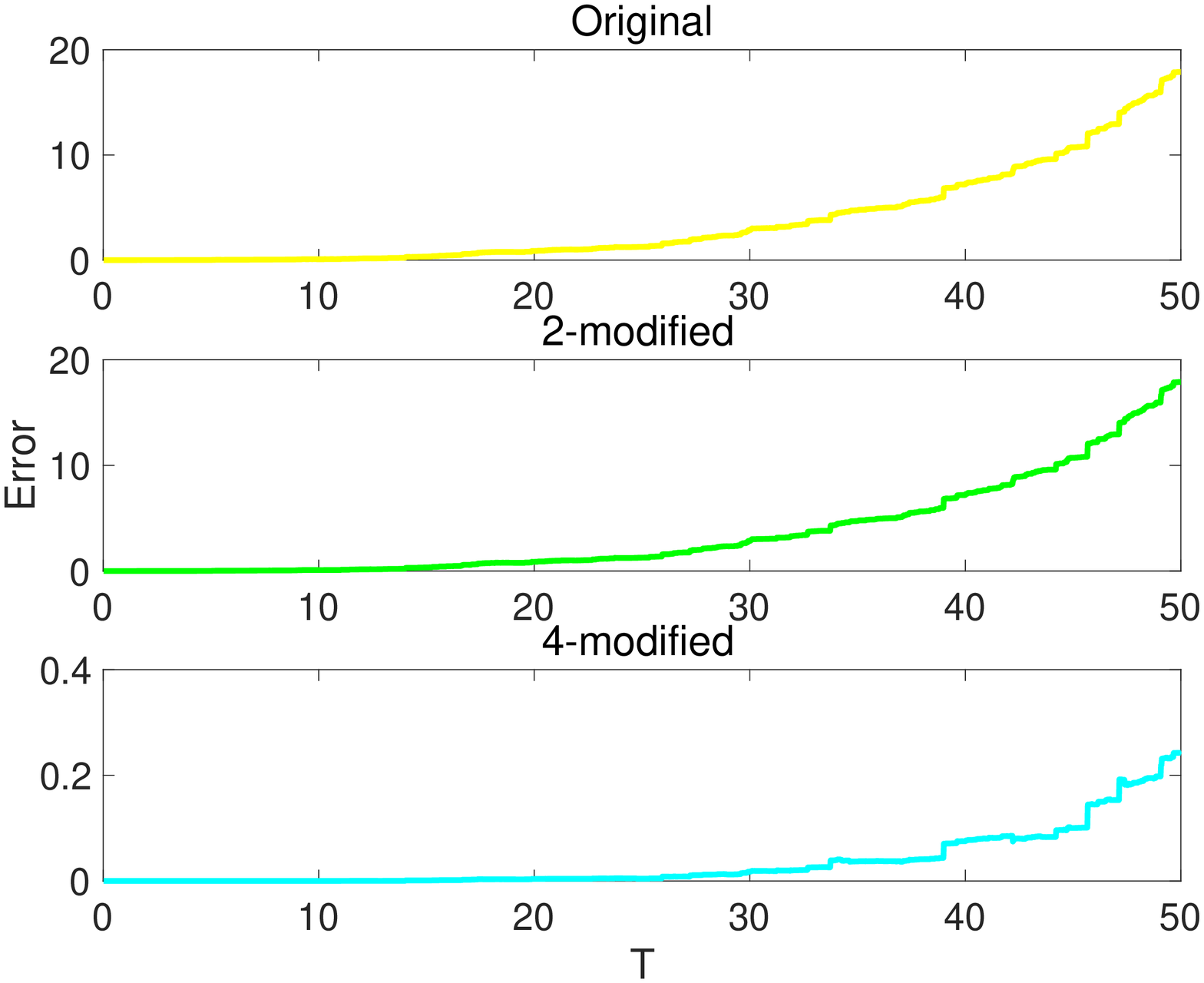}\label{a2}
		\end{minipage}
	}
	\subfigure[Energy error for one trajectory]{
		\begin{minipage}[t]{0.47\linewidth}
			\includegraphics[height=6cm,width=8cm]{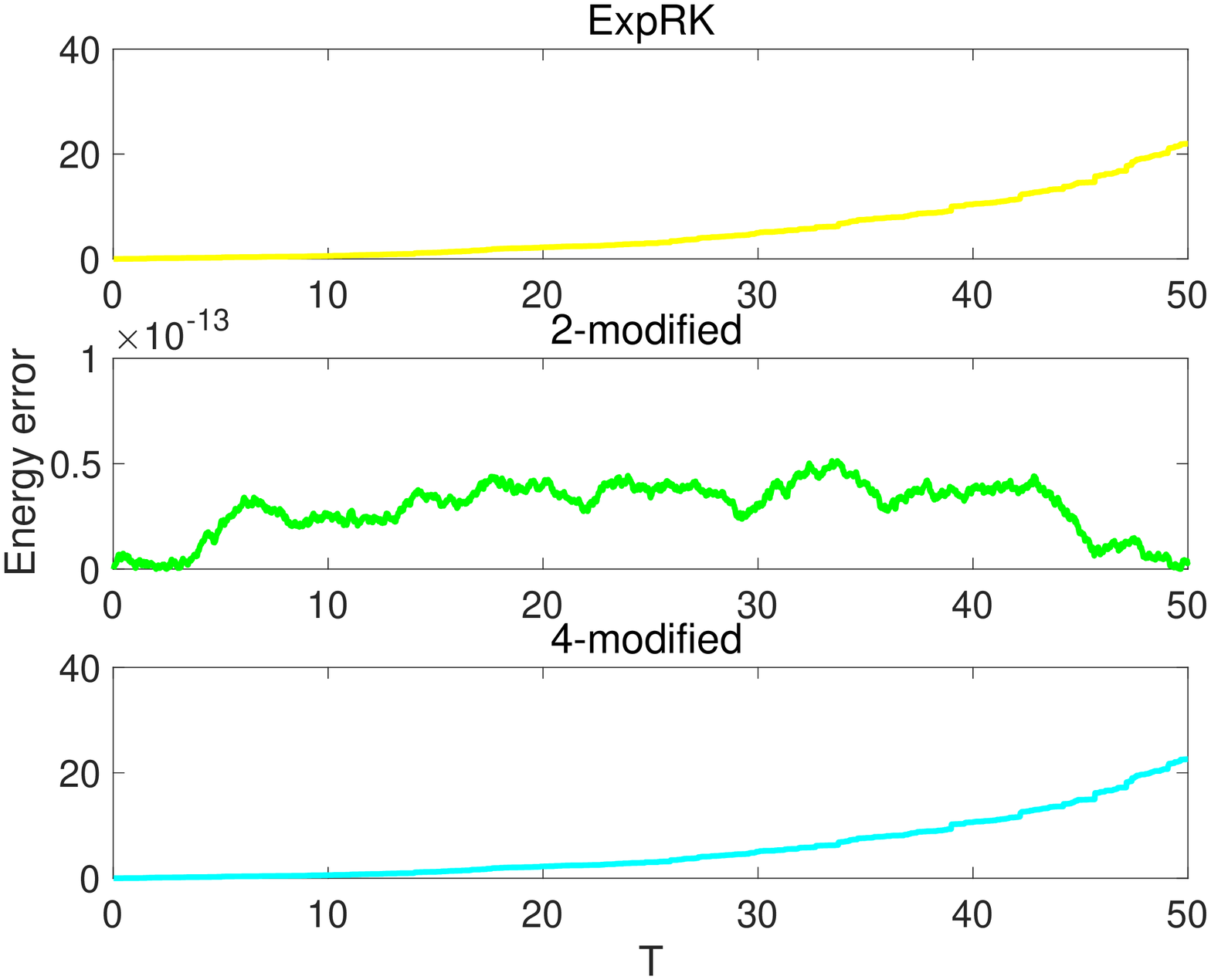}\label{b2}
		\end{minipage}
	}
	\caption{The explicit RK method \eqref{M2}}\label{err_M2}
\end{figure}

\begin{figure}
	\centering
	\subfigure[Error for one trajectory]{
		\begin{minipage}[t]{0.47\linewidth}
			\includegraphics[height=6cm,width=8cm]{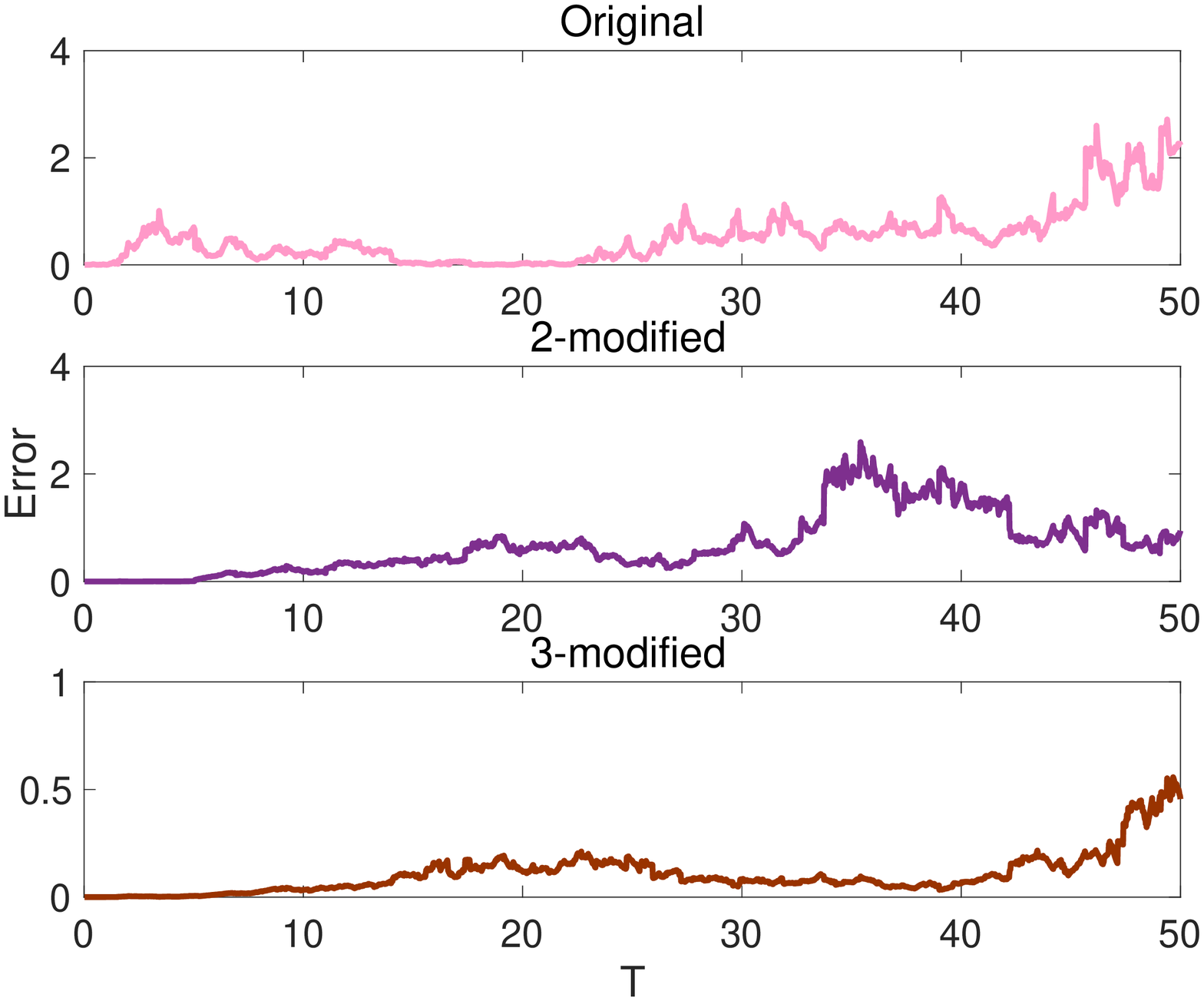}\label{a3}
		\end{minipage}
	}
	\subfigure[Energy error for one trajectory]{
		\begin{minipage}[t]{0.47\linewidth}
			\includegraphics[height=6cm,width=8cm]{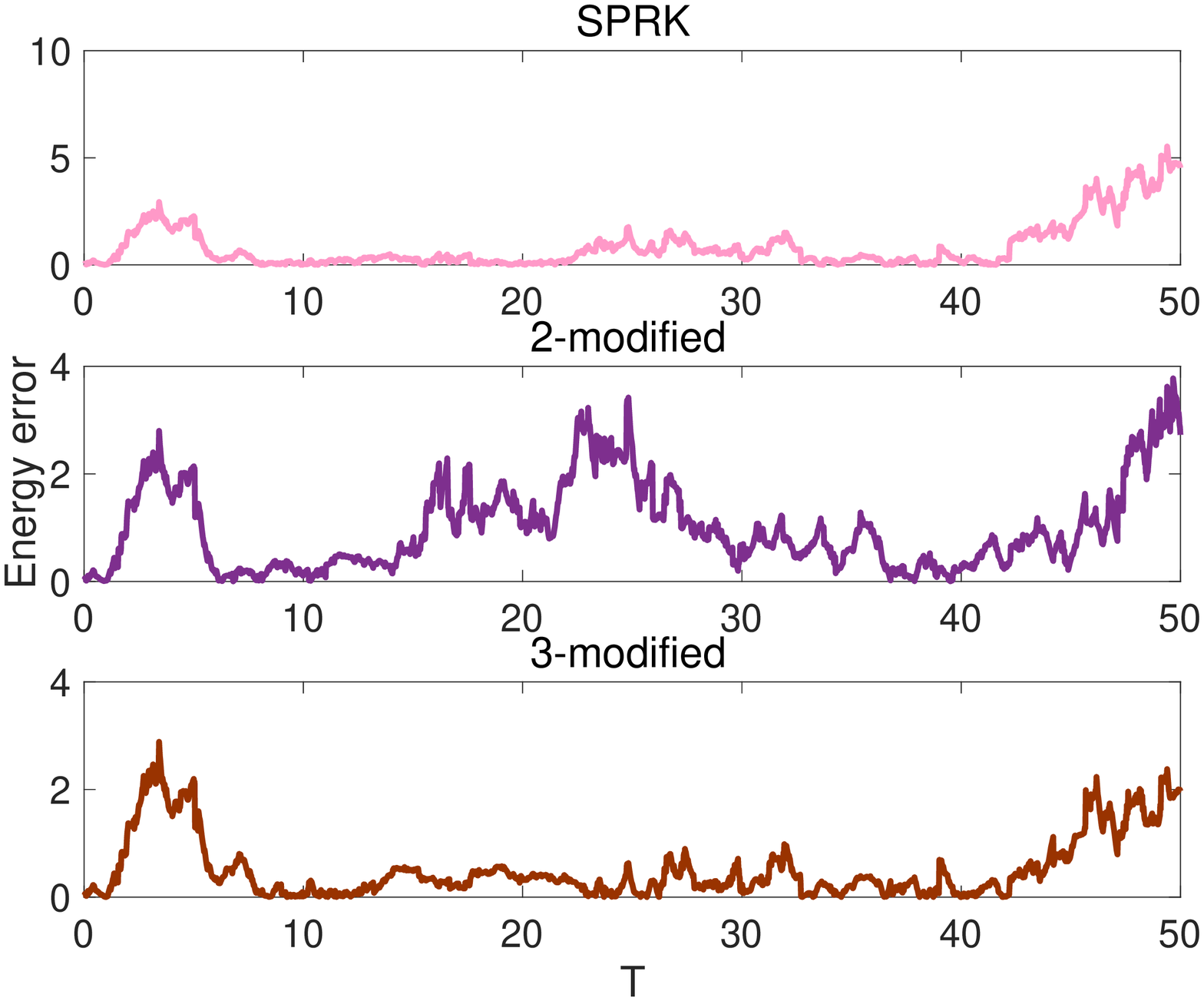}\label{b3}
		\end{minipage}
	}
	\caption{The symplectic partitioned RK method \eqref{M3}}\label{err_M3}
\end{figure}

\section*{Acknowledgements}
This work is supported by National Natural Science Foundation of China (NO. 91530118, NO. 91130003, NO. 11021101, NO. 91630312 and NO. 11290142).

\section*{References}

\bibliographystyle{plain}
\bibliography{mybibfile}

\end{document}